\documentclass[a4, 12pt, 
]{amsart}

\usepackage{amsthm}

\usepackage[normalem]{ulem} %
\usepackage{graphicx,xcolor}
\usepackage{dsfont}
\usepackage{enumitem}
\usepackage{tikz}
\usepackage{tikz-cd}
\usepackage[utf8]{inputenc}
\usepackage[T1]{fontenc}
\usepackage[all,cmtip]{xy}
\usepackage{todonotes}
\usepackage[titletoc]{appendix}
\usepackage{color}
\usepackage{amsmath,amssymb,amscd,amsfonts}
\usepackage{babel}
\usepackage{amstext}
\usepackage{amsmath}
\usepackage{amsfonts}
\usepackage{latexsym}
\usepackage{ifthen}
\usepackage{xypic}
\usepackage{yfonts}
\usepackage[all,cmtip]{xy}
\xyoption{all}
\usepackage{enumitem}

\usepackage{graphicx,xcolor}

\usepackage[pagebackref]{hyperref}
\hypersetup{
  colorlinks=true,
  linkcolor=magenta,
  urlcolor=red,
  citecolor=blue
}

\setlist[enumerate]{label=(\thethm.\arabic*), before={\setcounter{enumi}{\value{equation}}}, after={\setcounter{equation}{\value{enumi}}}}

\newcommand{\Ker}[0]{\operatorname{Ker}}

\newcommand{\Hom}[0]{\operatorname{Hom}}

\newcommand{\Supp}[0]{\operatorname{Supp}}

\newcommand{\rank}[0]{\operatorname{rank}}
\newcommand{\codim}[0]{\operatorname{codim}}
\newcommand{\Sym}[0]{\operatorname{Sym}}
\newcommand{\GL}[0]{\operatorname{GL}}

\newcommand{\e}{\varepsilon}
\newcommand{\I}[1]{\mathcal{I}(#1)}

\newcommand{\reg}{{\rm{reg}}}

\newcommand{\sat}{{\rm{sat}}}

\newcommand{\cal}[1]{\mathcal{#1}}
\newcommand{\bb}[1]{\mathbb{#1}}

\newcommand{\idelbar}{\sqrt{-1}\,\partial\bar\partial}

\newcommand{\R}{\mathbb{R}}

\newcommand{\ddbar}{\partial\bar{\partial}}

\newcommand{\cF}{\mathcal{F}}
\newcommand{\cE}{\mathcal{E}}
\newcommand{\orb}{\rm{orb}}
\newcommand{\tf}{\rm{tf}}

\newcommand{\cG}{\mathcal{G}}

\newcommand{\ep}{\varepsilon}
\renewcommand{\epsilon}{\varepsilon}

\renewcommand{\ker}{\mathrm{Ker} \,}

\renewcommand{\ge}{\geqslant}
\renewcommand{\le}{\leqslant}
\renewcommand{\leq}{\leqslant}
\renewcommand{\geq}{\geqslant}

\newcommand{\Imm}{\mathrm{Im}}
\newcommand{\Mov}{\mathrm{Mov}}
\newcommand{\Psef}{\mathrm{Psef}}

\newcommand{\C}{\mathbb C}

\newcommand{\rk}{\mathrm{rk}}

\newcommand{\dbar}{\bar \partial}

\DeclareMathOperator{\id}{id} %
\DeclareMathOperator{\tor}{tor} %

\newtheorem{thm}{Theorem}[section]
\newtheorem{lemma}[thm]{Lemma}
\newtheorem{proposition}[thm]{Proposition}
\newtheorem{question}[thm]{Question}

\newtheorem{cor}[thm]{Corollary}

\theoremstyle{remark}

\theoremstyle{definition}

\newtheorem{defn}[thm]{Definition}

\newtheorem{step}{Step}
\newtheorem{setup}[thm]{Setup}

\theoremstyle{definition} 
\newtheorem{case}{Case}
\numberwithin{equation}{thm}

\theoremstyle{remark}
\newtheorem{remark}[thm]{Remark}

\DeclareFontFamily{U}{MnSymbolC}{}
\DeclareSymbolFont{MnSyC}{U}{MnSymbolC}{m}{n}
\DeclareFontShape{U}{MnSymbolC}{m}{n}{
<-6>  MnSymbolC5
<6-7>  MnSymbolC6
<7-8>  MnSymbolC7
<8-9>  MnSymbolC8
<9-10> MnSymbolC9
<10-12> MnSymbolC10
<12->   MnSymbolC12}{}
\DeclareMathSymbol{\intprod}{\mathbin}{MnSyC}{'270}

\makeatletter
\newcommand*{\da@rightarrow}{\mathchar"0\hexnumber@\symAMSa 4B }
\newcommand*{\da@leftarrow}{\mathchar"0\hexnumber@\symAMSa 4C }
\newcommand*{\xdashrightarrow}[2][]{%
  \mathrel{%
    \mathpalette{\da@xarrow{#1}{#2}{}\da@rightarrow{\,}{}}{}%
  }%
}
\newcommand{\xdashleftarrow}[2][]{%
  \mathrel{%
    \mathpalette{\da@xarrow{#1}{#2}\da@leftarrow{}{}{\,}}{}%
  }%
}
\newcommand*{\da@xarrow}[7]{%
  \sbox0{$\ifx#7\scriptstyle\scriptscriptstyle\else\scriptstyle\fi#5#1#6\m@th$}%
  \sbox2{$\ifx#7\scriptstyle\scriptscriptstyle\else\scriptstyle\fi#5#2#6\m@th$}%
  \sbox4{$#7\dabar@\m@th$}%
  \dimen@=\wd0 %
  \ifdim\wd2 >\dimen@
    \dimen@=\wd2 %
  \fi
  \count@=2 %
  \def\da@bars{\dabar@\dabar@}%
  \@whiledim\count@\wd4<\dimen@\do{%
    \advance\count@\@ne
    \expandafter\def\expandafter\da@bars\expandafter{%
      \da@bars
      \dabar@ 
    }%
  }%
  \mathrel{#3}%
  \mathrel{%
    \mathop{\da@bars}\limits
    \ifx\\#1\\%
    \else
      _{\copy0}%
    \fi
    \ifx\\#2\\%
    \else
      ^{\copy2}%
    \fi
  }%
  \mathrel{#4}%
}
\makeatother

\usepackage{geometry}
\title[A flatness criteriaon for pseudo-effective sheaves]
{A flatness criterion \\
for pseudo-effective sheaves \\
on compact K\"ahler spaces}

\author{Junyan CAO}
\address{Laboratoire de Math\'ematiques J.A. Dieudonn\'e UMR 7351 CNRS, Universit\'e C\^{o}te d'Azur Parc Valrose 06108, Nice, France} 
\email{\tt junyan.cao@univ-cotedazur.fr}
	\email{\tt junyan.cao@unice.fr}

\author{Ya DENG}
\address{CNRS,  
		Institut de Math\'ematiques de Jussieu-Paris Rive Gauche,
		Sorbonne Universit\'e, Campus Pierre et Marie Curie,
		4 place Jussieu, 75252 Paris Cedex 05, France}
\email{\tt ya.deng@math.cnrs.fr}
\email{\tt deng@imj-prg.fr}

\author{Shin-ichi MATSUMURA}
\address{
Mathematical Institute 
$\&$ Division for the Establishment of Frontier Science of Organization for Advanced Studies, 
Tohoku University, 
6-3, Aramaki Aza-Aoba, Aoba-ku, Sendai 980-8578, Japan.}
\email{\tt mshinichi-math@tohoku.ac.jp}
\email{\tt mshinichi0@gmail.com}

\date{\today, version 0.01}



\begin{document}

\maketitle

\begin{abstract}
In this paper, we prove that if $\mathcal{E}$ is a pseudo-effective sheaf on a klt compact K\"ahler space $X$ 
whose first Chern class vanishes, 
then, after passing to a finite quasi-\'etale cover, the reflexive pullback of $\mathcal{E}$ is locally free and flat.
This extends the flatness criterion of H\"oring--Peternell, originally established for projective varieties, to the K\"ahler setting.

The proof relies on two main ingredients, both of which are new even in the projective case.
The first is a flatness theorem for stable sheaves:
we show that a slope-stable pseudo-effective sheaf with vanishing first Chern class is Hermitian flat.
This is obtained by combining Hermitian--Einstein theory with the subharmonicity properties of direct image sheaves.
The second is a singular K\"ahler analogue of Simpson's flatness theorem for extensions of locally free Hermitian flat sheaves.
\end{abstract}

\section{Introduction}\label{sec-intro}

In complex geometry, a fundamental theme is to extract a \textit{flat part} from ``semipositive'' curvature conditions
in order to understand the global structure of the underlying spaces.
A guiding principle, going back to the celebrated work of Demailly--Peternell--Schneider \cite{DPS94},
is that ``semipositively curved'' vector bundles with vanishing first Chern class should be flat.
This principle is one of the key ingredients in the structure theorem for compact K\"ahler manifolds with nef tangent bundle.

There are several ways to formulate semipositivity.
In the projective setting, positivity in the sense of Viehweg is one of the most useful notions of semipositivity,
especially in the classification theory of projective varieties.
For sheaves on normal compact K\"ahler spaces, the notion relevant to our purposes is \textit{pseudo-effectivity}
(see Definition~\ref{def-psef} and Remark~\ref{rem-hp}).
This is a weak positivity property which arises naturally in complex geometry,
particularly in the study of structure theorems for varieties with non-negative curvature;
see, for example, \cite{CH19, EIM23, GKP21, GKP22, LPT18, Mat22, MW25b, Wan22}, among many others.
Also, in recent years, the Minimal Model Program in the K\"ahler setting has also seen substantial progress
\cite{Das23, DH25, DHP24, DO24, Fuj22, GP25, HP16, HP18}.
These developments make it natural and important to extend the principle of Demailly--Peternell--Schneider
to pseudo-effective sheaves on singular K\"ahler spaces, and such an extension is expected to have various applications
(see \cite{Mat25a, MQ25, MWWZ25}).

The purpose of this paper is to establish two flatness criteria for pseudo-effective sheaves
on compact K\"ahler spaces with mild singularities; see Theorem~\ref{thm-main} and Theorem~\ref{thm-stable} below.

\begin{thm}\label{thm-main}
Let $X$ be a compact K\"ahler space with potentially klt singularities
$($i.e.\,there exists an effective $\mathbb{Q}$-divisor $\Delta$ such that $(X, \Delta)$ is a klt pair$)$.
Let $\mathcal{E}$ be a pseudo-effective sheaf on $X$ whose first Chern class vanishes.
Then the following assertions hold$:$
\begin{itemize}
\item[$\rm(i)$] There exists a finite quasi-\'etale cover $\nu \colon X' \to X$ such that the reflexive pullback
$\nu^{[*]}\mathcal{E}:=(\nu^{*}\mathcal{E})^{**}$ is a numerically flat locally free sheaf on $X'$.

\item[$\rm(ii)$] 
The restriction of the reflexive hull $\mathcal{E}^{**}$ to the smooth locus $X_{\reg}$ is a flat locally free sheaf on $X_{\reg}$.
\end{itemize}
\end{thm}

We refer to Definition~\ref{def-psef} for the definition of pseudo-effective sheaves,
and to Section~\ref{numflat} for the definitions of numerical flatness and Hermitian flatness.
Theorem~\ref{thm-main} extends the flatness criterion of H\"oring--Peternell \cite{HP19},
originally proved in the projective case, to the K\"ahler setting.
In the previous work \cite{HP19}, the variety $X$ is assumed to be smooth in codimension $2$,
whereas our result does not require this assumption
(see Remark~\ref{rem-hp} for a detailed comparison).

We briefly explain the proof of Theorem~\ref{thm-main}~(i).
One of the main difficulties is to prove the local freeness of $\mathcal{E}^{**}$.
First, by recent work \cite{Bra21, GKKP11, DHP24, Fuj22}, we may take a maximal quasi-\'etale cover
$\nu\colon X' \to X$
such that every flat locally free sheaf on the smooth locus $X'_{\reg}$ extends to a flat locally free sheaf on $X'$
(see Theorem~\ref{thm-maxqetale}).
The reflexive pullback $\nu^{[*]}\mathcal{E}$ is again pseudo-effective and satisfies
$c_1(\nu^{[*]}\mathcal{E})=0$.
By Proposition~\ref{semistable}, the sheaf $\nu^{[*]}\mathcal{E}$ is slope-semistable.
We may therefore consider the Jordan--H\"older filtration of $\nu^{[*]}\mathcal{E}$:
$$
0 =: \mathcal{E}_0 \subset \mathcal{E}_1 \subset \mathcal{E}_2 \subset \cdots \subset \mathcal{E}_k := \nu^{[*]}\mathcal{E}.
$$
By proving a regularity property for this filtration in Proposition~\ref{prop-ref},
we reduce the problem to showing that each graded quotient is Hermitian flat.
This is precisely the content of Theorem~\ref{thm-stable} below.

For the proof of Theorem~\ref{thm-main}~(ii), it remains to descend the flat connection on
$\nu^{[*]}\mathcal{E}$ to $X_{\reg}$.
This descent is not automatic in our singular setting and requires a variant of Simpson's theorem
on extensions of flat bundles \cite{Sim92}
(see Subsection~\ref{numflat}).
To prove this variant, we use the explicit construction of flat connections in \cite{Den21}.

\medskip

The next theorem gives a flatness criterion for pseudo-effective sheaves
on normal compact K\"ahler spaces, without assuming klt singularities.

\begin{thm}\label{thm-stable}
Let $X$ be a normal compact K\"ahler space, and let $\mathcal{E}$ be a pseudo-effective sheaf on $X$
whose first Chern class vanishes.
If $\mathcal{E}$ is slope-stable with respect to some K\"ahler class $\omega$,
then the restriction of the reflexive hull $\mathcal{E}^{**}$ to the smooth locus $X_{\reg}$ is Hermitian flat.
\end{thm}

Before explaining the idea of the proof of Theorem~\ref{thm-stable},
we recall two standard approaches to proving flatness.
The first is to deduce flatness from the stability of all reflexive symmetric powers $\Sym^{[m]} \cE$ 
rather than merely polystability, as in \cite{HP19, CCP19}. 
However, in our situation, the sheaves $\Sym^{[m]} \cE$ need not be stable.
The second approach is to prove the vanishing of the second Chern class and then apply a flatness criterion,
as in \cite{BS94, CCM21, LT18, GKP16b}; see also the recent developments in the K\"ahler setting \cite{FO25, GP25, GP26}.
However, it is not clear how to prove the vanishing of the second Chern class in the present setting,
even when $X$ has terminal singularities.

In the proof of Theorem~\ref{thm-stable}, we avoid using the second Chern class.
Instead, we develop an idea from \cite{CCP19}, namely using the theory of Hermitian--Einstein metrics together with the subharmonicity of direct image sheaves established in \cite{CCP19}. This approach works well in the singular setting and even in the non-K\"ahler setting.
The proof proceeds roughly as follows.
By the definition of pseudo-effective sheaves,
there exist singular metrics $\{h_{\e}\}_{\e >0}$ on $\mathcal{O}_{\mathbb{P}(\mathcal{E})}(1)$
whose curvature is bounded below by $-\e p^{*}\omega_{X}$.
Here $\omega_{X}$ is a K\"ahler form on $X$, and
$p\colon \mathbb{P}(\mathcal{E}) \to X$ is the natural projection with hyperplane bundle
$\mathcal{O}_{\mathbb{P}(\mathcal{E})}(1)$.
Since $\mathcal{E}$ is not locally free in general,
we have to treat $\mathbb{P}(\mathcal{E})$ with some care.
To this end, we establish a characterization of pseudo-effective sheaves
(see Theorem~\ref{thm-extension}),
which is highly nontrivial and is one of the main contributions of this paper.

In this introduction, we keep the following discussion at a heuristic level.
A key point is to study the singularities of a limit metric $h$ on $\mathcal{O}_{\mathbb{P}(\mathcal{E})}(1)$
obtained from the metrics $\{h_{\e}\}_{\e >0}$.
If the singular locus of $h$ is not dominant over $X$,
then the positivity of direct image sheaves implies that the $L^{2}$-metric $H$ on $\mathcal{E}$
has semipositive curvature, using the condition $c_{1}(\mathcal{E})=0$.
Hence, by Raufi's result \cite{Rau15},
the sheaf $\mathcal{E}$ is Hermitian flat and locally free on $X_{\reg}$.

It remains to exclude the opposite case, where the singular locus of $h$ is dominant over $X$.
We do so by constructing a destabilizing subsheaf of $\mathcal{E}$.
For simplicity, assume that $\det \mathcal{E}$ is a trivial invertible sheaf.
We consider the natural inclusion
\begin{align*}
\mathcal{S}:=
&p_{*} \big( \mathcal{O}_{\mathbb{P}(\mathcal{E})}(K_{\mathbb{P}(\mathcal{E})/X} + \mathcal{O}_{\mathbb{P}(\mathcal{E})}(r+1)) \otimes \I{h^{r+1}} \big) \\
\subset&
p_{*} \big( \mathcal{O}_{\mathbb{P}(\mathcal{E})}(K_{\mathbb{P}(\mathcal{E})/X} + \mathcal{O}_{\mathbb{P}(\mathcal{E})}(r+1)) \big) = \mathcal{E},
\end{align*}
where $\mathcal{I}(\bullet)$ is the multiplier ideal sheaf and $r$ is the rank of $\mathcal{E}$.
The $L^{2}$-metric on $\mathcal{S}$ is positively curved, and hence its slope is semipositive.
Therefore, we expect that $\mathcal{S}$, or a closely related direct image sheaf, gives a destabilizing subsheaf of $\mathcal{E}$,
contradicting the stability of $\mathcal{E}$.
The difficulty is that it is not clear a priori that $\mathcal{S}$ is nonzero and properly contained in $\mathcal{E}$.

To overcome this difficulty, we take a smooth metric $g$ on $\mathcal{O}_{\mathbb{P}(\mathcal{E})}(1)$
arising from Hermitian--Einstein metrics, as in \cite{CGNPPW},
and consider the subsheaf
$$
\mathcal{S}_{\delta}:=
p_{*} \big( \mathcal{O}_{\mathbb{P}(\mathcal{E})}(K_{\mathbb{P}(\mathcal{E})/X}) \otimes
\mathcal{O}_{\mathbb{P}(\mathcal{E})}(r+1) \otimes
\I{(h^{1-\delta} \cdot g^{\delta})^{r+1}} \big).
$$
By analyzing the singularities of the metric $h^{1-\delta} \cdot g^{\delta}$
on $\mathcal{O}_{\mathbb{P}(\mathcal{E})}(1)$ in Proposition~\ref{addlemma},
we find some $0<\delta<1$ such that $\mathcal{S}_{\delta}$ is nonzero and properly contained in $\mathcal{E}$.
Moreover, although the curvature of $g$ is not necessarily semipositive,
the slope of $\mathcal{S}_{\delta}$ is still non-negative
by the subharmonicity of direct image sheaves \cite{CCP19}.
This gives a destabilizing subsheaf, contradicting the stability of $\mathcal{E}$.

We emphasize that the argument explained above is new even in the projective setting.
In the projective setting, the flatness criteria of this type, as used for instance in \cite{CH19}, have played an important role
in the study of singular analogues of the Beauville--Bogomolov--Yau decomposition.
Our flatness criterion may provide a more direct proof of such decomposition theorems in the K\"ahler setting. 
Also, one promising application of our flatness criterion is to extend the Beauville--Bogomolov--Yau decomposition
for klt pairs established in \cite{MW25b} to compact K\"ahler spaces.
This is carried out in forthcoming work, which will appear soon.

\medskip
Finally, in Section~\ref{nonkahler}, we apply the techniques developed above to study numerically flat vector bundles
on compact complex manifolds which are not necessarily K\"ahler.
In particular, we extend an important characterization of numerically flat vector bundles on compact K\"ahler manifolds
due to Demailly--Peternell--Schneider \cite[Theorem 1.18]{DPS94}
(see also Theorem~\ref{thm-stable}) to the non-K\"ahler setting.
This gives a positive answer to the question raised in \cite[Remark 1.21]{DPS94}.

\begin{thm}\label{numflatnonkahler}
Let $X$ be a compact complex manifold and let $E$ be a locally free sheaf on $X$.
Assume that $E$ is numerically flat, namely $c_1 (E)=0$ and $\mathcal{O}_{\mathbb P (E)} (1)$ is nef.
Then $E$ admits a filtration by holomorphic subbundles
\begin{equation}\label{filtra1}
\{0\}=E_0\subset E_1\subset\cdots\subset E_k=E
\end{equation}
such that each graded quotient $E_i/E_{i-1}$ is Hermitian flat.
In particular, all Chern classes $c_k(E)$ vanish.
\end{thm}

\subsection*{Acknowledgements.}\label{subsec-ack}
This collaboration began at the 2025 Summer Research Institute in Algebraic Geometry,
and the authors thank the organizers of the relevant sessions, Professors Philippe Eyssidieux, Mihai P\u{a}un, and Christian Schnell.
S.\,M. thanks  Professors Shiyu Zhang and Xiaojun Wu for discussions on Proposition \ref{prop-ref} and
maximal quasi-\'etale covers.
He also thanks  Professors Junayong Wang, Xiaojun Wu, and Qiming Zhang 
for discussions on Lemma \ref{thm-extension}. 
J.\,C. thanks the Institut Universitaire de France for providing
excellent working conditions and Professor Indranil Biswas for raising the question that led to Theorem~\ref{numflatnonkahler}.

S.\,M. was partially supported by JSPS Grant-in-Aid for Scientific Research (B) $\sharp$21H00976 and by the JST FOREST Program $\sharp$PMJFR2368.
J. \,C and Y. \,D acknowledge the support of the French Agence Nationale de la Recherche (ANR)
under reference ANR-21-CE40-0010. 
The authors used generative AI tools solely for grammatical correction and to improve the fluency of the English; all mathematical content was written and verified by the authors.

\section{Tools and Preparations}\label{sec-pre}

In this section, we summarize some basic material and preliminary results.
Throughout this paper, let $X$ be a normal analytic variety of dimension $n$
(i.e.\ an irreducible and reduced complex analytic space with at worst normal singularities),
and let $\mathcal{E}$ be a torsion-free coherent sheaf on $X$ of rank $r$.
Set $X_{0}:=X_{\reg} \cap X_{\mathcal{E}}$,
where $X_{\reg}$ is the smooth locus of $X$ and $X_{\mathcal{E}}$ is the maximal locus on which $\mathcal{E}$ is locally free.
We use the terms \emph{locally free sheaf} and \emph{vector bundle} interchangeably, 
and we reserve the term \emph{vector bundle} for occasions when we want to emphasize the bundle viewpoint.

\subsection{Notation and conventions}\label{subsec-notation}

For a positive integer $m\in\mathbb{Z}_+$, we define
\begin{itemize}
\item the reflexive tensor power $\mathcal{E}^{[\otimes m]} :=(\mathcal{E}^{\otimes m})^{**}$;
\item the reflexive symmetric power $\textup{Sym}^{[m]}\mathcal{E} :=(\textup{Sym}^m\mathcal{E})^{**}$;
\end{itemize}
where $\mathcal{E}^*:= \Hom(\mathcal{E}, \mathcal{O}_X)$ is the dual sheaf.
The reflexive top exterior power $\det\mathcal{E}:= \wedge^{[r]}\mathcal{E}$
is called the \textit{determinant sheaf} of $\mathcal{E}$.
Further, for a torsion-free sheaf $\mathcal{F}$ on $X$,
we define the reflexive tensor product $\mathcal{E} [\otimes]\mathcal{F} := (\mathcal{E} \otimes\mathcal{F} )^{**}$.

Let $f \colon Y \rightarrow X$ be a morphism between normal analytic varieties.
We define the reflexive pullback $f^{[*]}\mathcal{E} := (f^{*}\mathcal{E})^{**}$.
In this paper, we often say that a property (P) holds \textit{over} a subset $U \subset X$ if it holds on the inverse image $f^{-1}(U)$.
For example, we say that the sheaf $\mathcal{G}$ is locally free \textit{over} $X_{0}$ if it is locally free on the inverse image $f^{-1}(X_{0})$.

\subsection{Singular Hermitian metrics and pseudo-effective sheaves}\label{subsec-psef}

We first recall the notion of psh (resp.\ quasi-psh) functions on a normal analytic variety $X$.
See \cite{Dem85} and \cite[Section 4.6]{BG13} for details.

A function $\varphi\colon X \to [-\infty, +\infty[$ is said to be \textit{plurisubharmonic} (psh for short) if there exists an open covering $X=\bigcup_i U_i$ and embeddings $U_i \hookrightarrow \mathbb{C}^N$ such that $\varphi|_{U_i}$ is the restriction of a psh function defined on an open subset of $\mathbb{C}^N$.
A function $\varphi\colon X \to [-\infty, +\infty[$ is said to be \textit{quasi-psh}
if it can be locally written as the sum of a psh function and a smooth function.
Similarly, a $(1,1)$-form $\omega_{X}$ on $X$ is said to be a \textit{K\"ahler form} on $X$ if $\omega_{X}|_{U_i}$ can be written as the restriction of a K\"ahler form defined on an open subset of $\mathbb{C}^N$.
A K\"ahler form $\omega_{X}$ in our sense always admits local potentials
(i.e.\ locally there exists a smooth function $f$ such that $\omega_{X}=\sqrt{-1}\ddbar f$).
Psh functions on normal varieties satisfy the Hartogs-type extension property,
which will often be used in this paper without explicit mention.
\begin{thm}[{\cite[Satz 4]{GR56}}]
\label{thm-gr}
Let $X$ be a normal analytic variety and let $Z \subset X$ be a Zariski closed subset of codimension $\geq 2$.
Then, any psh function $\varphi$ on $X\setminus Z$ uniquely extends to a psh function on $X$.
\end{thm}

Using quasi-psh functions,
we review singular Hermitian metrics and pseudo-effectivity for torsion-free sheaves.
Below, we assume familiarity with the basic theory of singular Hermitian metrics on vector bundles (see \cite{Rau15, HPS18, Pau18, PT18}).
We recommend \cite[Sections 16--19]{HPS18} for a detailed exposition.

A singular Hermitian metric $h$ on a torsion free coherent sheaf $\cal{E}$ is a (possibly singular) Hermitian metric on the vector bundle $\mathcal{E}|_{X_{0}}$,
where $X_{0}:=X_{\reg}\cap X_{\mathcal{E}}$.
In particular, outside a subset of measure zero in $X$,
the metric $h(x)$ defines a finite Hermitian inner product on $\cE_x$.
For a smooth $(1,1)$-form $\theta$ on $X$ with local potential,
we write
$$
\sqrt{-1}\Theta_{h} \geq \theta \otimes \id \text{ on } X
$$
if, for any local section $e \in H^0(U, \mathcal{E}^*)$ on an open set $U \subset X$,
the function $\log |e|_{h^{*}} - f$ is psh on $U \cap X_0$,
where $f$ is a local potential of $\theta$ and $h^{*}$ is the induced metric on the dual sheaf
$\mathcal{E}^*$. 
Note that a singular Hermitian metric is only defined on $X_{0}$, even if it is smooth.
Nevertheless, the function $\log |e|_{h^{*}}$ extends to a quasi-psh function across $X \setminus X_{0}$ by Theorem \ref{thm-gr}.

The following definition extends the notion of pseudo-effectivity for vector bundles to torsion-free sheaves.
\begin{defn}[{\cite[Definition 2.1]{Mat23}}]
\label{def-psef}
Let $X$ be a normal compact K\"ahler space and let $\omega_{X}$ be a K\"ahler form on $X$.
A torsion-free sheaf $\mathcal{E}$ on $X$ is said to be \textit{pseudo-effective} if, for every $m \in \mathbb{Z}_{+}$,
there exists a singular Hermitian metric $h_{m}$ on the $m$-th symmetric power $\Sym^{m} \cal{E}|_{X_{0}}$
such that
$$
\sqrt{-1}\Theta_{h_{m}}\geq -\omega_{X} \otimes \id.
$$
\end{defn}

\begin{remark}\label{rem-hp}
Suppose that $X$ is a projective variety with an ample line bundle $A$.

(1) We compare our definition of pseudo-effective sheaves with Viehweg-type positivity in the projective setting.
By \cite[Proposition~2.4]{Mat23} (see also \cite[Theorem~2.21]{Pau18}),
the pseudo-effectivity of $\mathcal{E}$ is equivalent to the following condition:
for every $a\in \mathbb{Z}_{+}$, there exists $b\in \mathbb{Z}_{+}$ such that the reflexive hull
$\Sym^{[ab]}\mathcal{E}\otimes (bA)$ is globally generated at a general point of $X$.

(2) In \cite[Theorem 2.1]{HP19}, H\"oring--Peternell proved a flatness result for $\mathcal{E}$
when $\mathcal{E}$ is pseudo-effective in the sense of \cite[Definition 2.1]{HP19}
and $\Sym^{[m]}\mathcal{E}$ is stable for every $m$.
Here, pseudo-effectivity in the sense of \cite[Definition 2.1]{HP19} means that
for every $c \in \mathbb{Z}_{+}$ there exist integers $i,j \in \mathbb{Z}_{+}$ such that $i>cj$ and
$$
H^{0}\bigl(X, \Sym^{[i]}\mathcal{E} \otimes A^{\otimes j}\bigr) \neq 0.
$$
This definition is weaker than ours, even when $X$ is smooth and $\mathcal{E}$ is locally free.
Our Definition \ref{def-psef} implies that the non-nef locus of $\mathcal{O}_{\mathbb{P}(\mathcal{E})}(1)$ is not dominant over $X$.
\end{remark}

The following propositions will be frequently used in this paper.
\begin{proposition}\label{genersurjet}
Let $X$ be a normal analytic space and let
$$
\cE \to \cF
$$
be a generically surjective morphism between torsion-free sheaves $\cE$ and $\cF$.
If a Hermitian metric $h$ on $\mathcal{E}$ satisfies
$$
\sqrt{-1}\Theta_{h} \geq \theta \otimes \id \text{ on } X,
$$
then the induced metric $h_{q}$ on $\cF$ satisfies the same condition.
In particular, if $\cE$ is pseudo-effective, then so is $\cF$.
\end{proposition}

\begin{proof}
A key point is that $\cE \to \cF$ is only generically surjective, and not necessarily surjective everywhere.
The proof follows from the argument of \cite[Lemma 2.4.3]{PT18}.
\end{proof}

\begin{proposition}\label{dualpsef}
Let $\mathcal{E}$ be a torsion-free sheaf on a normal compact K\"ahler space $X$.
Let $\Gamma^{a}(\mathcal{E})$ be the reflexive hull of $\Gamma^{a}(\mathcal{E}|_{X_{0}})$,
where $\Gamma^{a}(\mathcal{E}|_{X_{0}})$ is the locally free sheaf associated with a tensor representation of highest weight
$a=(a_{1}, \cdots, a_{r})$ with $a_{1}\geq  \cdots \geq a_{r} \geq 0$.
Assume that $\mathcal{E}$ is pseudo-effective.
Then $\Gamma^{a}(\mathcal{E})$ is pseudo-effective.
In particular, the reflexive exterior power $\Lambda^{[m]}\mathcal{E}$ is pseudo-effective.
Moreover, if $c_1(\cE)=0$, then $\cE^*$ is also pseudo-effective.
\end{proposition}
\begin{proof}
The proof follows from the argument of \cite[Proposition 1.14]{DPS94},
and thus, we briefly explain only the main strategy.
Take singular Hermitian metrics $\{h_{m}\}$ on $\Sym^{m}\mathcal{E}|_{X_{0}}$
such that $\sqrt{-1}\Theta_{h_{m}}\geq -\omega_{X} \otimes \id$.
After replacing $X$ with $X_{0}$, we may assume that $\mathcal{E}$ is locally free.

Recall that $\Sym^{m}(\Gamma^{a}(\mathcal{E}))$ is a direct summand of $\Gamma^{b}(\mathcal{E})$ with $|b|=m|a|$, and that
$\Gamma^{b}(\mathcal{E})$ is a direct summand of $\Sym^{b_{1}}(\mathcal{E})\otimes \cdots \otimes \Sym^{b_{r}}(\mathcal{E})$.
On the other hand, the product metric $h_{b_{1}}\otimes \cdots \otimes h_{b_{r}}$ satisfies
$$
\sqrt{-1}\Theta_{h_{b_{1}}\otimes \cdots \otimes h_{b_{r}}}\geq -r\,\omega_{X} \otimes \id.
$$
Note that, although $h_{m}$ is singular, the tensor product metric can be defined by approximating $h_{m}$ locally by smooth metrics via convolution.
By Proposition \ref{genersurjet},
the induced quotient metric $g_{m}$ on $\Sym^{m}(\Gamma^{a}(\mathcal{E}))$ satisfies the same inequality.
This shows that $\Gamma^{a}(\mathcal{E})$ is pseudo-effective.

For the last assertion, assume that $c_1(\cE)=0$.
Then $\Gamma^a(\cE)\otimes (\det \cE)^*$ is pseudo-effective on $X$ by (see Proposition \ref{prop-flat}).
In particular, taking $a=(1,0,\ldots,0)$, we obtain that $\cE^*$ is pseudo-effective.
\end{proof}

\subsection{Stabilities of torsion-free sheaves}\label{subsec-stable}

The determinant sheaf $\det\mathcal{E}$ is reflexive, but it is not necessarily invertible when $X$ has singularities.
Nevertheless, we can define the first Chern class $c_{1}(\mathcal{E})$ as a functional on the classes of $d$-closed forms as follows: 
Take a modification $\pi \colon \tilde{X} \to X$ such that
$\tilde{X}$ is smooth and $(\pi^{*}\mathcal{E})/\tor$ is a locally free sheaf,
where $\tor$ is the maximal torsion subsheaf of $\pi^{*}\mathcal{E}$.
Then, for a $d$-closed $(n-1, n-1)$ form $\gamma$, 
the first Chern class $c_{1}(\mathcal{E})$ is defined by 
$$
\int_{X}  c_1 (\mathcal{E} ) \wedge  \gamma:=
\int_{\tilde{X} } c_1 \bigl((\pi^{*}\mathcal{E})/\tor\bigr) \wedge \pi^* \gamma 
$$
This definition is independent of the choice of $\pi \colon \tilde{X} \to X$.
As in the smooth case, the slope (and hence stability) is defined as follows:

\begin{defn}
Let $X$ be a normal compact K\"ahler space and let $\omega_X$ be a K\"ahler form on $X$.
Let $\mathcal{E}$ be a torsion-free sheaf on $X$, and let
$\pi \colon \tilde{X} \to X$ be a desingularisation of $X$ such that $(\pi^{*}\mathcal{E})/\tor$ is a locally free sheaf.
Then, the slope of $\mathcal{E}$ with respect to $\omega_X$ is defined by
\begin{align*}
\mu_{\omega_X} (\cE) :&=
\frac{1}{\rank \cE}\int_{X}  c_1 (\mathcal{E} ) \wedge  \omega_X ^{n -1}\\
&:=
\frac{1}{\rank \cE}\int_{\tilde{X} } c_1 \bigl((\pi^{*}\mathcal{E})/\tor\bigr) \wedge \pi^* \omega_X ^{n -1}.
\end{align*}
\end{defn}
The right-hand side is independent of the choice of desingularisation,
since 
$$\int_E \pi^* \omega_X ^{n -1} =0$$ 
for any $\pi$-exceptional divisor $E$
for dimensional reasons.

As a consequence of Propositions \ref{genersurjet} and \ref{dualpsef}, we have the following:
\begin{proposition}\label{semistable}
Let $X$ be a normal compact K\"ahler space and let $\mathcal{E}$ be a pseudo-effective sheaf on $X$.
Let $\omega_X$ be a K\"ahler form on $X$. If $c_1 (\cE)=0$, then $\cE$ is semistable with respect to $\omega_{X}$.
\end{proposition}
\begin{proof}
Let $\mathcal{S} \subset \mathcal{E}$ be a saturated subsheaf, and set $\mathcal{Q}:=\mathcal{E}/\mathcal{S}$.
Then $\mathcal{Q}$ is pseudo-effective by Proposition \ref{genersurjet},
and so is $\det \mathcal{Q}$ by Proposition \ref{dualpsef}.
This shows that the slope of $\mathcal{S}$ is semipositive.
\end{proof}

\subsection{Maximal quasi-\'etale covers in the K\"ahler setting}\label{subsec-max}

Following \cite[Section 3.2]{Gue25} and \cite[Subsection 2.B]{FO25}, we summarize some results on maximally quasi-\'etale covers in the analytic setting.
Note that \cite[Corollary 3.6]{Gue25} and \cite[Theorem 2.5]{FO25} are formulated for klt complex analytic varieties, but the same argument applies to potentially klt singularities.

\begin{defn}
Let $X$ be a normal analytic variety. A maximally quasi-\'etale cover
of $X$ is a quasi-\'etale Galois cover $X' \to X$ such that the naturally induced morphism
\[
\iota_{*}\colon {\hat{\pi}_{1}}(X'_{\mathrm{\reg}})\longrightarrow {\hat{\pi}_{1}}(X')
\]
of \'etale fundamental groups is an isomorphism.
\end{defn}

\begin{thm}[{\cite[Theorem 1.5]{GKP16}, \cite[Cor 3.6]{Gue25}, \cite[Theorem 2.5]{FO25}}]
\label{thm-maxqetale}
Let $X$ be a potentially klt compact analytic variety.
Then, we have$:$

$(1)$ There exists a maximally quasi-\'etale cover $\nu \colon X'\to X$ of $X$.

$(2)$ Suppose that a normal analytic variety $X$ is maximally quasi-\'etale $($i.e.\,$X$ itself satisfies the above property$)$.
Then, any linear representation $\rho_0 \colon \pi_{1}(X_{\reg}) \rightarrow \GL(r, \mathbb{C})$ factors through $\pi_{1}(X)$:
\begin{equation*}
\xymatrix@C=40pt@R=30pt{
\pi_{1}(X_{\reg}) \ar@/^18pt/[rr]^{\rho_0} \ar@{->>}[r]_{i_{*}} & \pi_1(X) \ar@{->}[r]_{\rho\quad} & \GL(r, \mathbb{C}).\\ 
 }
 \end{equation*}
\end{thm}
\begin{proof}
As explained in \cite[Corollary 3.6]{Gue25} and \cite[Subsection 2.B]{FO25}, the first conclusion follows
from the finiteness of the local fundamental group of $X_{\reg}$
thanks to the arguments in \cite{Bra21, Xu14}, where the existence of a plt blowup plays a crucial role via \cite{BCHM10}.
This has recently been established using developments of the Minimal Model Program for projective morphisms in the analytic setting \cite{DHP24, Fuj22}.

The second conclusion follows from Malcev's theorem, as in \cite{GKP16}. We briefly explain the proof for the reader's convenience.
Note that $\pi_{1}(X_{\reg})$ is finitely generated.
Except for this point, the statement is purely group-theoretic, so we set
\[
G:=\pi_{1}(X_{\reg}),\quad H:=\pi_{1}(X),
\quad i_{*}\colon G\to H. 
\]
Let $\widehat{G}$ (resp.\ $\widehat{H}$) be the profinite completion of $G$ (resp.\ $H$).
Consider the commutative diagram
\[
\xymatrix{
G \ar[r]^{i_*} \ar[d]_{p_G} & H \ar[d]^{p_H} \\
\widehat{G} \ar[r]^{\widehat{i}_*} & \widehat{H}. 
}
\]
Since $X$ is maximally quasi-\'etale, the bottom arrow $\widehat{i}_*$ is an isomorphism.

It is enough to show that
$
\ker(i_{*})\subset \ker(\rho_{0}).
$
Let $\gamma \in \ker(i_{*})$ and assume for contradiction that $\rho_{0}(\gamma)\neq 1$.
Set $\Gamma:=\rho_{0}(G)\subset \GL(r,\C)$.
Since $\Gamma$ is a finitely generated linear group, Malcev's theorem shows that $\Gamma$ is residually finite.
Therefore, there exists a homomorphism $\psi\colon \Gamma\to F$
to some finite group $F$ such that $\psi(\rho_{0}(\gamma))\neq 1$.
Then $\Ker (\psi\circ \rho_{0}) \triangleleft G$ is a normal subgroup of $G$
of finite index and does not contain $\gamma$.
This means that $p_{G}(\gamma) \neq 1$,
which contradicts the commutative diagram.
\end{proof}

\subsection{Numerically flat locally free sheaves}\label{numflat}

We first recall three related notions of flatness.

\begin{defn}
Let $X$ be a normal analytic variety and let $\mathcal{E}$ be a locally free sheaf on $X$.
\begin{itemize}
\item[(1)] The sheaf $\mathcal{E}$ is said to be \textit{flat} if it is constructed by a linear representation
$\pi_{1}(X) \to \GL(r, \mathbb{C})$.
If $X$ is smooth, this is equivalent to saying that $\mathcal{E}$ admits a holomorphic connection with vanishing curvature.

\item[(2)] The sheaf $\mathcal{E}$ is said to be \textit{Hermitian flat} if it admits a smooth Hermitian metric with vanishing Chern curvature.

\item[(3)] Suppose that $X$ is compact.
The sheaf $\mathcal{E}$ is said to be \textit{numerically flat} if both $\mathcal{E}$ and $\mathcal{E}^*$ are nef.
Equivalently, $\mathcal{E}$ is nef and $c_1(\mathcal{E})=0$.
\end{itemize}
\end{defn}

Note that these notions are defined only for locally free sheaves.
We assume that $X$ is compact when discussing numerical flatness, while (1) and (2) are also defined for non-compact varieties.
Recall that $\mathcal{E}$ is \textit{numerically effective} (nef for short) 
if $\mathcal{O}_{\mathbb P (\mathcal{E})}(1)$ is a nef line bundle on $\mathbb P (\mathcal{E})$.
Demailly--Peternell--Schneider in \cite{DPS94} proved
that $\mathcal{E}$ is nef if and only if, for every $m \in \mathbb{Z}_{+}$, the $m$-th symmetric power $\Sym^{m}\mathcal{E}$
admits a smooth Hermitian metric $h_{m}$ such that
$$
\sqrt{-1}\Theta_{h_{m}}\geq -\omega_{X} \otimes \id.
$$
Moreover, they established a flatness criterion for numerically flat locally free sheaves:

\begin{thm}[{\cite[Theorem 1.18]{DPS94}}]\label{DPS18}
Let $X$ be a compact K\"ahler manifold and let $\mathcal{E}$ be a locally free sheaf on $X$.
If $\mathcal{E}$ is numerically flat, then $\mathcal{E}$ admits a filtration of holomorphic subbundles
\begin{equation}\label{filtra}
\{0\}= E_0 \subset E_1 \subset \cdots \subset E_k =\mathcal{E}
\end{equation}
such that each graded quotient $E_i /E_{i-1}$ is Hermitian flat.
\end{thm}

Combining Theorem \ref{DPS18} with Simpson's result \cite{Sim92}
(see also \cite{Den21} for an explicit construction),
we obtain the following: 

\begin{cor}\label{flatcon}
Let $X$ be a compact K\"ahler manifold and let $\mathcal{E}$ be a locally free sheaf on $X$.
If $\mathcal{E}$ is numerically flat, then $\mathcal{E}$ admits a holomorphic flat connection
which is compatible with the filtration \eqref{filtra}.
\end{cor}

Therefore, Theorem \ref{thm-main} can be viewed as a generalization of Theorem \ref{DPS18} and Corollary \ref{flatcon} to the singular setting:
the sheaf $\cE$ need not be locally free, the metrics $h_m$ may be singular, and the base space $X$ itself may also be singular.

\section{Technical results}\label{subse-technical}

In this section, we prepare several auxiliary propositions.
The first one, although technical, generalizes \cite[Proposition 1.16]{DPS94} to compact K\"ahler spaces.

\begin{proposition}\label{prop-ref}
Let $X$ be a compact K\"ahler space with potentially klt singularities.
Let $\mathcal{E}$ be a pseudo-effective reflexive sheaf on $X$ with vanishing first Chern class.
Consider an exact sequence
$$
0 \to \mathcal{S} \to \mathcal{E} \to \mathcal{Q} \to 0 \quad \text{ on } X,
$$
where $\mathcal{S}$ is reflexive and $\mathcal{Q}$ is torsion-free with vanishing first Chern class.
Then, there exists a Zariski closed subset $Z\subset X$ of codimension $\geq 3$ such that
$\mathcal{Q}$ is reflexive on $X\setminus Z$.
\end{proposition}

\begin{proof}
Set $X_{0}:=X_{\reg} \cap X_{\mathcal{E}}$. 
By \cite[Lemma 2.9]{IMZ26},
the quotient $\mathcal{Q}$ is locally free, and the sequence in the proposition is an exact sequence of vector bundles on $X_{0}$.
However, this is not enough, since $\codim (X \setminus X_{0}) \geq 3$ does not hold in general.
The proof is divided into four steps.

\begin{step}
Consider the reflexive sheaf
$$
\mathcal{F}:=\Lambda^{[s]} \mathcal{E} [\otimes] (\det \mathcal{Q})^{*},
$$
where $s$ is the rank of $\mathcal{Q}$.
Let $\tau \in H^0(X, \cF^*)$ be the natural section obtained from the inclusion
$
0 \to \mathcal{O}_{X} \to \mathcal{F}^{*}
$
induced by the exact sequence in the proposition.
Since $\mathcal{E}$ is pseudo-effective and $c_{1}(\mathcal{Q})=0$, 
there exist singular Hermitian metrics $\{h_{m}\}$ on $\Sym^{m}\mathcal{F}$
such that
$$
\sqrt{-1}\Theta_{h_{m}}(\Sym^{m}\mathcal{F})
\geq - \omega_{X} \otimes \id \quad\text{ on } X,
$$
where $\omega_{X}$ is a fixed K\"ahler form on $X$ (see Proposition \ref{prop-flat}).
The function
$
\log |\tau^{m}|_{h^{*}_{m}}
$
satisfies
$$
\sqrt{-1}\ddbar \Big( \frac{1}{m}\log |\tau^{m}|_{h^{*}_{m}} \Big) \geq \frac{1}{m} \omega_{X}
$$
on $X_{0}$,
and hence extends to a quasi-psh function on $X$ with the same curvature inequality,
since $\codim (X\setminus X_{0}) \geq 2$ and $X$ is normal.
Thus, after passing to a suitable normalization and a subsequence,
we may assume that the functions
$(1/m) \log |\tau^{m}|_{h^{*}_{m}}$ converge to a psh function on $X$, which is constant since $X$ is compact.
\end{step}
\begin{step}
Since $X$ has potentially klt singularities,
we can take a Zariski closed subset $Z_1\subset X$ of codimension $\geq 3$
such that $X\setminus Z_1$ has only quotient singularities.
Thus, for a sufficiently small open set $U \subset X \setminus Z_{1}$,
we can take a finite quasi-\'etale cover $\nu\colon V \to U$
such that $V$ is smooth.
Since $V$ is smooth and $\nu^{[*]} \mathcal{E}$ is reflexive,
there exists a Zariski closed subset $Z \subset U$ of codimension $\geq 3$
such that the reflexive pullback
$$
\mathcal{E}_{\orb}:=\nu^{[*]} \mathcal{E}
$$
is locally free on $U \setminus Z$.
Although $Z$ is defined only locally on $U$,
the non-reflexive locus of $\mathcal{Q}$ is a globally defined Zariski closed subset of $X$.
Thus, for our purposes, it suffices to prove that $\mathcal{Q}$ is reflexive on $U \setminus Z$.
\end{step}

\begin{step}
Let $\mathcal{S}_{\sat}$ be the saturation of $\nu^{[*]}\mathcal{S} \subset \nu^{[*]}\mathcal{E}$, and consider
\begin{align}\label{eq-orb}
0 \to \mathcal{S}_{\sat} \to \mathcal{E}_{\orb}=\nu^{[*]}\mathcal{E} \to \mathcal{Q}_{\tf}:= \mathcal{E}_{\orb}/\mathcal{S}_{\sat} \to 0 \quad\text{ on } V.
\end{align}
Note that $\mathcal{Q}_{\tf}$ is torsion-free on $V$.
As in Step~1, we obtain the natural section $\tau_{\orb} \in H^0(V, \mathcal{G}^*)$ induced by \eqref{eq-orb},
where
$$
\mathcal{G}:=\Lambda^{[s]} \mathcal{E}_{\orb}[\otimes] (\det \mathcal{Q}_{\tf})^{*}.
$$
In this step, we prove that $\tau_{\orb}$ is nowhere vanishing over $U \setminus Z$. Note that if $\tau_{\orb}$ is nowhere vanishing over $U \setminus Z$, then by \cite[Lemma 1.20]{DPS94}, both $\mathcal{S}_{\sat}$ and $\mathcal{Q}_{\tf}$ are locally free over $U\setminus Z$.

Consider singular Hermitian metrics $\{h_{\orb, m}\}$ on $\Sym^{m}\mathcal{G}$
defined by pulling back $\{h_{m}\}$, namely $h_{\orb, m}:=\nu^{*}h_{m}$.
More precisely, the metric $h_{m}$ is a priori defined only on $X_{\reg} \cap X_{\Sym^{m}\mathcal{F}}$.
Hence the pullback $h_{\orb, m}$ is defined on $V_1:=\nu^{-1}(X_{\reg} \cap X_{\Sym^{m}\mathcal{F}})$
via $\nu^{*} \big( \Sym^{m}\mathcal{G} \big)=\Sym^{m}\mathcal{F}$.
By $\codim (V\setminus V_1 )\geq 2$,
the metric $h_{\orb, m}$ extends to a singular Hermitian metric on $\Sym^{m}\mathcal{G}$.
By construction, we have
$$
\tau_{\orb}= \nu^{*}\tau \quad\text{and}\quad
\log |\tau^{m}_{\orb}|_{h_{\orb, m}^{*}}=\nu^{*}\bigl( \log |\tau^{m}|_{h_{m}^{*}}\bigr)
\quad\text{ on }  V_1.
$$
The function $\log |\tau^{m}_{\orb}|_{h_{\orb, m}^{*}}$ is quasi-psh on $V_1$, and hence extends across $V\setminus V_{1}$.
On the other hand, the function $\log |\tau^{m}|_{h_{m}^{*}}$ also extends to $X$.
Therefore, we obtain
\begin{align}\label{eq-pull}
\log |\tau^{m}_{\orb}|_{h_{\orb, m}^{*}}=\nu^{*}\bigl( \log |\tau^{m}|_{h_{m}^{*}}\bigr)
\quad\text{ on } V.
\end{align}

Assume for contradiction that $\tau_{\orb}(p)=0$ for some point $p \in \nu^{-1}(U\setminus Z)$.
To consider Lelong numbers on smooth varieties,
we take a resolution of singularities $\pi\colon \tilde{X} \to X$ and a smooth variety $\tilde{V}$
fitting into the commutative diagram
\[
\xymatrix{
\tilde{V} \ar[r]^{\tilde{\nu}\quad } \ar[d]^{\tilde{\pi}} & \pi^{-1}(U) \ar@{^{(}->}[r]  & \tilde{X} \ar[d]^{\pi} \\
V  \ar[r]^{\nu}     & U \ar@{^{(}->}[r] & X.
}
\]
Take a point $\tilde{p} \in \tilde{\pi}^{-1}(p)$.
Since $\mathcal{G}$ is locally free around $p$,
the Lelong number
$$
\frac{1}{m}\nu\bigl(\tilde{\pi}^{*}  (\log |\tau^{m}_{\orb}|_{h_{\orb, m}^{*}}), \tilde{p}\bigr)
=
\frac{1}{m}\nu\bigl( \log |\tilde{\pi}^{*} \tau^{m}_{\orb}|_{\tilde{\pi}^{*} h_{\orb, m}^{*}}, \tilde{p}\bigr)
$$
is bounded from below uniformly in $m$.
Indeed, writing $\tau^{m}_{\orb}$ locally as
$$
\tau^{m}_{\orb}=\sum_I \tau_I e_I,
$$
each holomorphic function $\tau_I$ vanishes to order $\geq m$ at $p$ by $\tau_{\orb}(p)=0$.
Here $\{e_i\}_{i=1}^r$ is a local frame of $\mathcal{G}$, $I$ is a multi-index of degree $m$, and $e_I:=\prod_{i\in I} e_i$.
Moreover, since $\log |u|_{h_{\orb, m}^{*}}$ is quasi-psh for any local section $u$,
the inner product $\langle e_I, e_J \rangle_{h_{\orb, m}^{*}}$ is locally bounded from above.
Consequently, there exists a constant $C>0$ such that
$$
|\tau^{m}_{\orb}|_{h_{\orb, m}^{*}} \leq C \sum_I |\tau_I|.
$$
This implies that the Lelong number of $(1/m)\log |\tau^{m}_{\orb}|_{h_{\orb, m}^{*}}$ at $p$ is at least $1$.

By \cite[Theorem 2]{Fav99},  the Lelong number
$$
\frac{1}{m}\nu\bigl(\pi^{*} \log |\tau^{m}|_{h_{m}^{*}}, \tilde{\nu}(\tilde{p})\bigr)
$$
is also bounded from below uniformly in $m$.
However, the function $(1/m)\pi^{*} \log |\tau^{m}|_{h_{m}^{*}}$ converges to a constant by \eqref{eq-pull},
which is a contradiction.
\end{step}

\begin{step}
After taking a Galois closure, we may assume that $\nu \colon V \to U$ is a Galois cover with Galois group $G$.
By Step~3, both $\mathcal{S}_{\sat}$ and $\mathcal{Q}_{\tf}$ are locally free over $U\setminus Z$.
In particular, $\mathcal{S}_{\sat}$, $\mathcal{E}_{\orb}$, and $\mathcal{Q}_{\tf}$ are reflexive over $U\setminus Z$.
Hence, their $G$-invariant pushforwards $(\nu_{*}(\bullet))^{G}$ are also reflexive over $U\setminus Z$ by \cite[Lemma~A.3]{GKKP11}.
Thus, since $\mathcal{S}_{\sat}$, $\mathcal{E}_{\orb}$, and $\mathcal{Q}_{\tf}$
are the pullbacks of $\mathcal{S}$, $\mathcal{E}$, and $\mathcal{Q}$ over $U_{\reg}\cap U_{\mathcal{E}}$,
we obtain
$$
(\nu_{*}(\mathcal{S}_{\sat}))^{G} = \mathcal{S}^{**}=\mathcal{S}, \quad
(\nu_{*}(\mathcal{E}_{\orb}))^{G} = \mathcal{E}^{**}=\mathcal{E}, \quad
(\nu_{*}(\mathcal{Q}_{\tf}))^{G} = \mathcal{Q}^{**}.
$$
on $U\setminus Z$. 
Together with the exactness of the functor $(\nu_{*}(\bullet))^{G}$ (see \cite[Lemma A.3]{GKKP11}),
we deduce
$$
0 \to \mathcal{S} \to \mathcal{E} \to \mathcal{Q}^{**} \to 0
\quad\text{ on } U\setminus Z.
$$
This shows that $\mathcal{Q} \cong \mathcal{Q}^{**}$ on $U\setminus Z$.
This completes the proof.
\end{step}
\end{proof}

The next proposition analyzes the singularities of the limit metric on $\mathcal{O}_{\mathbb{P}(\mathcal{E})}(1)$
induced by Hermitian metrics on $\mathcal{E}$.

\begin{proposition}\label{addlemma}
Let $\{H_{\e}\}$ be a sequence of smooth Hermitian metrics on $\mathbb{C}^{n+1}$,
and let $h_{\e}$ be the induced metrics on $\mathcal{O}_{\mathbb P^{n}} (1)$ via the quotient map
$\mathbb  P^{n} \times \mathbb{C}^{n+1} \to \mathcal{O}_{\mathbb P^{n}} (1)$.

$(1)$ After applying a unitary change of coordinates on $\mathbb{C}^{n+1}$, rescaling $H_{\e}$, and passing to a subsequence,
the metrics $\{h_{\e}\}$ converge to a singular  metric $h$ on $\mathcal{O}_{\mathbb P^{n}} (1)$
whose weight can be written as
$$
\log \sum_{i=0}^{n} \lambda_i |\zeta_i|^2,
$$
where $[\zeta_{0}\colon \cdots \colon\zeta_{n}]$ are the homogeneous coordinate on $\mathbb{P}^{n}$
and $\{\lambda_{i}\}$ are real numbers with 
$\lambda_0 \geq \cdots \geq \lambda_n \geq 0$.

\smallskip
$(2)$ Suppose that the limit metric $h$ is singular $($i.e.\ $\lambda_n = 0$$)$.
Let $g$ be a smooth metric on $\mathcal{O}_{\mathbb P^{n}} (1)$.
Then, there exists $0<\delta <1$ such that the space of $L^{2}$-integrable sections
$$
H^0 \big(\mathbb P^n,  K_{\mathbb P^n} \otimes \mathcal{O}_{\mathbb P^n} (n+2) \otimes \mathcal{I}
\big( (g^{(1-\delta)} \cdot h^{\delta})^{(n+2)}\big)\big)
$$
is nonzero and is a proper subspace of $H^0 (\mathbb P^n,  K_{\mathbb P^n} \otimes \mathcal{O} (n+2))$.
\end{proposition}

\begin{remark}\label{rem-add}
(1) The second conclusion does not hold for an arbitrary singular metric
$h$ on $\mathcal{O}_{\mathbb P^n} (1)$.
A key point in the proposition is that $h$ arises as a limit of metrics induced by Hermitian metrics on $\mathbb{C}^{n+1}$.
Indeed, in the one-dimensional case (i.e.\,$n=1$),
consider the singular  metric $h$ on $\mathcal{O}_{\mathbb{P}^1} (1)$ such that
$$
\sqrt{-1}\Theta_{h}=\frac{1}{3}([p] + [q] + [r]),
$$
where $p, q, r \in \mathbb{P}^1$ are distinct points.
Then, in the case $\delta=1$, the line bundle
$$
K_{\mathbb P^1} \otimes \mathcal{O}_{\mathbb P^1} (n+2)\otimes \mathcal{I}
\big( (g^{(1-1)} \cdot h^{1})^{3}\big)
\cong \mathcal{O}_{\mathbb P^1} (-2)
$$
has no nonzero sections.
Moreover, for any $1>\delta>0$,
the space of $L^{2}$-integrable sections is not a proper subspace of the ambient space,
by $\mathcal{I}\big( (g^{(1-\delta)} \cdot h^{\delta})^{3}\big)=\mathcal{O}_{X}$.

(2) The proposition concerns limit metrics arising from the family $\{H_{\e}\}$ in the definition of weakly positively curved sheaves.
From this viewpoint, it is natural to study limit metrics appearing in the definition of pseudo-effective sheaves.
Although this is not used in this paper, we formulate the following question.
\begin{question}\label{prob-limit}
Let $\{H_{m}\}$ be smooth Hermitian metrics on the $m$-th symmetric power $\Sym^{m}(\mathbb{C}^{n+1})$,
and let $\{h_{m}\}$ be the  metrics on $\mathcal{O}_{\mathbb{P}^{n}} (m)$
induced by the surjective morphism 
$$\mathbb{P}^{n} \times \Sym^{m}\mathbb{C}^{n+1} \to \mathcal{O}_{\mathbb{P}^{n}} (m). $$
Can we describe the limit metric $h$ obtained from $\{h_{m}^{1/m}\}$, or the singularities of $h$?
\end{question}
\end{remark}

\begin{proof}[Proof of Proposition \ref{addlemma}]

(1)
Let $G$ be the standard Hermitian metric on $\mathbb{C}^{n+1}$ and let $(e_0,\dots,e_n)$ be the standard basis.
Then, there exists a unitary transformation $U_{\e}\in U(n+1)$ such that
$(U_{\e}(e_0),\dots,U_{\e}(e_n))$ is an orthogonal basis with respect to $H_{\e}$.
Therefore, we can write
\[
G=\sum_{i=0}^{n} e_{i}^{*} \otimes \overline{e_{i}^{*}}
=\sum_{i=0}^{n} U_{\e}(e_{i})^{*} \otimes \overline{U_{\e}(e_{i})^{*}},
\quad
H_{\e}= \sum_{i=0}^{n} \lambda_{i,\e}\, U_{\e}(e_{i})^{*} \otimes \overline{U_{\e}(e_{i})^{*}},
\]
where $\lambda_{0,\e}\ge \cdots \ge \lambda_{n,\e}>0$ are the eigenvalues of $H_{\e}$ with respect to $G$.
Replacing $H_{\e}$ with $(1/\lambda_{0,\e})\cdot H_{\e}$, we may assume that $\lambda_{0,\e}=1$.

Let $h_{\e}$ be the induced metric on $\mathcal{O}_{\mathbb{P}^{n}}(1)$.
After passing to a subsequence, we may assume that $\lambda_{i,\e}\to\lambda_i\in[0,1]$ for all $i$,
and that $U_{\e}\to U$, since the unitary group is compact.
Then $h_{\e}$ converges (in the sense of weights in $L^{1}_{\mathrm{loc}}$)
to a singular metric on $\mathcal{O}_{\mathbb{P}^{n}}(1)$, which we denote by $h$.
More precisely, for $\zeta=(\zeta_0,\dots,\zeta_n)\in\mathbb{C}^{n+1}\setminus\{0\}$,
let $[\,\zeta\,]\in\mathbb{P}^n$ be the corresponding point.
If we set $\eta_{\e}:=U_{\e}^{-1}\zeta$, then the weight function of $h_{\e}$ is given by
\[
\varphi_{\e}([\,\zeta\,])=\log\Bigl(\sum_{i=0}^{n}\lambda_{i,\e}\,|\eta_{\e,i}|^{2}\Bigr)
=\log\Bigl(\sum_{i=0}^{n}\lambda_{i,\e}\,|(U_{\e}^{-1}\zeta)_i|^{2}\Bigr).
\]
Hence $\varphi_{\e}$ converges to
\[
\varphi([\,\zeta\,])=\log\Bigl(\sum_{i=0}^{n}\lambda_{i}\,|(U^{-1}\zeta)_i|^{2}\Bigr),
\]
which defines the limit metric $h$ on $\mathcal{O}_{\mathbb{P}^{n}}(1)$.
Finally, after changing homogeneous coordinates by the unitary transformation $U$,
the weight of $h$ can be written as $\log \sum_{i=0}^{n}\lambda_i|\zeta_i|^{2}$.

\smallskip

(2)
For the second conclusion, after applying a unitary change of coordinates,
we may assume that $G$ is the standard Hermitian inner product, that $g$ is induced by $G$, and that
$h$ can be written as in {\rm (1)}.
Fix a basis $(e_0,\dots,e_n)$ of
$$
\mathbb C^{n+1} \cong H^0 (\mathbb P^n,  \mathcal{O}_{\mathbb P^n} (1)) \cong
H^0 (\mathbb P^n,  K_{\mathbb P^n} \otimes \mathcal{O}_{ \mathbb P^n}(n+2)),
$$
which induces the homogeneous coordinates in {\rm (1)}.

Define a singular  metric $H_\delta$ on $\mathcal{O}_{\mathbb P^n}(n+2)$ by
$$
H_\delta:= (g^{(1-\delta)} \cdot h^{\delta})^{(n+2)}.
$$
By assumption, there exists an integer $0<p<n$ such that
$$
\lambda_0 \geq \cdots \geq \lambda_{p} > \lambda_{p+1}= \cdots =\lambda_n =0.
$$
A straightforward computation shows that, for any $0 \leq k \leq n$, we have
\begin{align*}
\|e_0 \|^2_{ H_\delta} &= \int_{\mathbb C^{n}}
\frac{1}{
 (1+\sum_{i=1}^{n} |z_i|^2)^{(1-\delta) (n+2)} \cdot (\lambda_{0}+\sum_{i=1}^{p}  \lambda_{i} |z_i|^2 )^{\delta (n+2)} } \, dV\\
\|e_k\|^2_{ H_\delta} &= \int_{\mathbb C^{n}}
\frac{|z_{k}|^{2}}{
 (1+\sum_{i=1}^{n} |z_i|^2)^{(1-\delta) (n+2)} \cdot (\lambda_{0}+\sum_{i=1}^{p}  \lambda_{i} |z_i|^2 )^{\delta (n+2)} } \, dV,
\end{align*}
where $dV$ denotes the Lebesgue measure on $\C^n$.

From these expressions, we obtain the following integrability conditions:
\begin{itemize}
\item $\|e_k \|^2_{ H_\delta} < \infty$ for $0 \leq k \leq p$ if and only if $(1-\delta)(n+2) > (n -p)$.
\item $\|e_k \|^2_{ H_\delta} < \infty$ for $p+1 \leq k \leq n$ if and only if $(1-\delta)(n+2) > (n -p +1)$.
\end{itemize}
Choose $0<\delta<1$ such that $n -p +1 \geq (1-\delta)(n+2) > n -p$.
Then, the space
$$
H^0 (\mathbb P^n,  K_{\mathbb P^n} \otimes \mathcal{O}_{\mathbb P^n} (n+2) \otimes \mathcal{I}
\big( (g^{(1-\delta)} \cdot h^{\delta})^{(n+2)}\big)\big)
$$
is nonzero, since it contains the subspace spanned by $\{e_{i}\}_{i=0}^{p}$,
and it is properly contained in $H^0 (\mathbb P^n,  K_{\mathbb P^n} \otimes \mathcal{O}_{\mathbb P^n} (n+2))$
since it does not contain $e_{p+1},\dots,e_n$.

Checking these integrability conditions is straightforward.
For the reader's convenience, we verify the second one.
(The first one can be proved in the same way.)
It is enough to check integrability as $|z|\to\infty$.
Since $\lambda_0,\dots,\lambda_p$ are positive constants, they do not affect integrability,
and we may assume $\lambda_i=1$ for simplicity.
Set 
$$\text{
$a:=(n+2)(1-\delta)$, \quad $b:=(n+2)\delta$, \quad  $q:=n-p$.
}
$$
Let us consider the integrability of
\[
I:=\int_{\C^{n}}
\frac{|z_{p+1}|^{2}}
{\bigl(1+\sum_{i=1}^{n}|z_{i}|^{2}\bigr)^{a}\,
 \bigl(1+\sum_{i=1}^{p}|z_{i}|^{2}\bigr)^{b}}
\,dV. 
\]
Write $z=(z',z'')\in \C^p\times \C^q$ with
\[
z'=(z_1,\dots,z_p),\quad z''=(z_{p+1},\dots,z_n),
\]
and set $r:=\|z'\|$, $s:=\|z''\|$ so that $\|z\|^{2}=r^{2}+s^{2}$.
Using polar coordinates on $\C^p\cong \R^{2p}$ and $\C^q\cong \R^{2q}$, we write
\[
z'=r\,\xi,\quad (r\in[0,\infty),\ \xi\in S^{2p-1}),
\quad
z''=s\,\omega,\quad (s\in[0,\infty),\ \omega\in S^{2q-1}).
\]
The Euclidean volume forms can be written as
\[
dV_{\C^p}=r^{2p-1}\,dr\,d\sigma_p(\xi),
\quad
dV_{\C^q}=s^{2q-1}\,ds\,d\sigma_q(\omega),
\]
where $d\sigma_p$ and $d\sigma_q$ denote the standard sphere measures on
$S^{2p-1}$ and $S^{2q-1}$, respectively.

Since $|z_{p+1}|^2$ depends only on $z''$, by writing $z_{p+1}=s\omega_1$ and using Fubini's theorem, we obtain
\begin{align*}
I
&=\int_0^\infty\!\!\int_0^\infty\!\!\int_{S^{2p-1}}\!\!\int_{S^{2q-1}}
\frac{s^2|\omega_1|^2}{(1+r^2+s^2)^{a}(1+r^2)^{b}}\,
r^{2p-1}s^{2q-1}\,d\sigma_q(\omega)\,d\sigma_p(\xi)\,ds\,dr\\
&=C
  \int_0^\infty\!\!\int_0^\infty
\frac{r^{2p-1}s^{2q+1}}{(1+r^2+s^2)^{a}(1+r^2)^{b}}\,ds\,dr,
\end{align*}
where $C:=\bigl(\int_{S^{2p-1}}d\sigma_p\bigr)\bigl(\int_{S^{2q-1}}|\omega_1|^2\,d\sigma_q\bigr)>0$.
Thus it suffices to consider
\begin{align*}
I'
&:=\int_0^\infty\!\!\int_0^\infty
\frac{r^{2p-1}s^{2q+1}}{(1+r^2+s^2)^{a}(1+r^2)^{b}}\,ds\,dr \\
&=\int_0^\infty   \frac{r^{2p-1}}{(1+r^2)^{b}}\,dr
\Bigl( \int_0^\infty   \frac{s^{2q+1}}{(1+r^2+s^2)^{a}}\,ds \Bigr). 
\end{align*}

Changing variables $s=\sqrt{1+r^2}\,t$, we have
\[
1+r^2+s^2=(1+r^2)(1+t^2),\quad
ds=\sqrt{1+r^2}\,dt,\quad
s^{2q+1}=(1+r^2)^{q+\frac12}t^{2q+1}.
\]
Hence, we obtain 
\[
\int_0^\infty   \frac{s^{2q+1}}{(1+r^2+s^2)^{a}}\,ds
=(1+r^2)^{q+1-a}\, \int_0^\infty \frac{t^{2q+1}}{(1+t^2)^{a}}\,dt.
\]
Therefore, we have 
\begin{align*}
I'
&=\Bigl(\int_0^\infty \frac{t^{2q+1}}{(1+t^2)^{a}}\,dt\Bigr) \cdot 
\Bigl( \int_0^\infty r^{2p-1}(1+r^2)^{q+1-a-b}\,dr \Bigr)\\
&=\Bigl(\int_0^\infty \frac{t^{2q+1}}{(1+t^2)^{a}}\,dt\Bigr)\cdot 
\Bigl(  \int_0^\infty r^{2p-1}(1+r^2)^{q+1-(n+2)}\,dr \Bigr), 
\end{align*}
where we used $a+b=n+2$.

Now the first integral converges if and only if $a>q+1$, i.e.\ $a>n-p+1$,
since its integrand behaves like $t^{2q+1-2a}$ as $t\to\infty$.
The second integral always converges, since its integrand behaves like
$r^{2p-1+2(q+1-(n+2))}=r^{2(p+q-(n+1))-1}$ as $r \to \infty$.
Thus $I<\infty$ holds if and only if $a>n-p+1$,
which is exactly the second condition.
\end{proof}

\section{Proofs of the main results}\label{sec-proof}
\subsection{Characterization of pseudo-effective sheaves}\label{subsec-pre}

In this subsection, we study a torsion-free sheaf $\mathcal{E}$ on a normal compact analytic variety $X$. 
We then establish Theorem \ref{thm-psef}, 
which gives a characterization of the pseudo-effectivity of $\mathcal{E}$ in terms of the hyperplane bundle 
$\mathcal{O}_{\bb{P}(\mathcal{E})}(1)$. 
As an application, we give a criterion for the condition $c_{1}(\det \mathcal{E})=0$ in Proposition \ref{prop-flat}. 

We first fix the notation used throughout this subsection. 

\begin{setup}\label{setup}
Let $\mathcal{E}$ be a torsion-free sheaf on a normal compact analytic variety $X$.
Let $\pi_{\mathcal{E}} \colon  \bb{P}(\mathcal{E}) \to X$
denote the normalization of the irreducible component of
the projectivization ${\bf{Proj}}(\bigoplus_{m=0}^{\infty}\operatorname{Sym}^{m} \mathcal{E})$ 
of the graded sheaf
$\bigoplus_{m=0}^{\infty}\operatorname{Sym}^{m} \mathcal{E}$ that dominates $X$. 
Let $\pi \colon P \to \bb{P}(\mathcal{E})$
be a desingularization of $\bb{P}(\mathcal{E})$ fitting into the commutative diagram: 
\begin{align}\label{eq-comm}
\xymatrix{
\bb{P}(\mathcal{E}) \ar[d]_{\pi_{\mathcal{E}}}&  & P\ar[ll]_{\hspace{0cm} \pi}  \ar@/^16pt/[dll]^{p} \\
X. &  &
}
\end{align}
Set $X_{0} := X_{\reg} \cap X_{\mathcal{E}}$ and $P_{0} := p^{-1}(X_{0})$.
Note that $\codim (X \setminus X_{0}) \geq 2$.
We may assume that $\pi \colon P \to \bb{P}(\mathcal{E})$ is an isomorphism over $X_{0}$ 
so that $p \colon P_{0} \to X_{0}$ can be identified with the projective space bundle
$\bb{P}(\mathcal{E}|_{X_{0}}) \to X_{0}$ associated with the locally free sheaf $\mathcal{E}|_{X_{0}}$.
We use the following notation:
\begin{itemize}
  \item[$\bullet$] $L := \pi^{*} \mathcal{O}_{\bb{P}(\mathcal{E})}(1)$,
    where $\mathcal{O}_{\bb{P}(\mathcal{E})}(1)$ is the pullback of the hyperplane bundle on
    ${\bf{Proj}}(\bigoplus_{m=0}^{\infty} \operatorname{Sym}^{m} \mathcal{E})$ to the normalization of the main component;
  \item[$\bullet$] $\omega_{P}$ is a K\"ahler form on $P$;
  \item[$\bullet$] $\omega_{X}$ is a K\"ahler form on $X$;
  \item[$\bullet$] $\Lambda$ is an effective $p$-exceptional divisor such that
    \[
    p_{*}\mathcal{O}_{P}(m(L + \Lambda)) = \operatorname{Sym}^{[m]}\mathcal{E}.
    \]
    See \cite[\S III.5, Lemma 5.10, pp.~107--108]{Nak04} for the existence of such a divisor $\Lambda$.
\end{itemize}
After replacing $P$ by a further blow-up, we may assume that
$$
\text{
$P \setminus P_{0}$ is divisorial and the support of $\Lambda$ coincides with $P \setminus P_{0}$.
}
$$
\end{setup}

Then, we have the following theorem. 

\begin{thm}\label{thm-psef}
In the notation of Setup \ref{setup}, 
the sheaf $\mathcal{E}$ is pseudo-effective $($in the sense of Definition \ref{def-psef}$)$
if and only if 
the line bundle $L+\Lambda$ admits 
singular metrics $\{h_{\e}\}_{\e>0}$ satisfying the following conditions$:$
\begin{itemize}
\item[$(1)$] $\sqrt{-1}\Theta_{h_{\e}} + \e p^{*}\omega_{X} \geq \delta_{\e} \omega_{P}$ holds on $P$ 
for some $\delta_{\e}>0$$;$ 
\item[$(2)$] $\{x \in P\,|\, \nu(h_{\e}, x) >0 \}$ does not dominate $X$, 
where $\nu(h_{\e}, x)$ denotes the Lelong number of the local weight of $h_{\e}$. 
\end{itemize}
\end{thm}
By Remark \ref{rem-psef}, 
the line bundle $L$ does not admit the desired metrics $\{h_{\e}\}_{\e>0}$ without adding $\Lambda$.
The ``if'' part follows from the positivity of direct image sheaves  
and a similar argument is given in the proof of Theorem \ref{thm-stable}. 
Therefore, we focus on proving the ``only if'' part, which depends on the following two lemmas.

\begin{lemma}\label{lemm-metric}
In the notation of Setup \ref{setup}, assume that $\mathcal{E}$ is pseudo-effective. 
Then, the line bundle $L|_{P_{0}}$ admits 
singular Hermitian metrics $\{h_{\e}\}_{\e>0}$ 
satisfying the same conditions in Theorem \ref{thm-psef}, 
with $P$ and $X$ replaced by $P_{0}$ and $X_{0}$, respectively. 
\end{lemma}
\begin{proof}
We first show that there exists a singular metric $h$ on $L$
such that $h$ is smooth on $P_0$ and satisfies
\[
\sqrt{-1}\Theta_h + C\, p^{*} \omega_X \geq a\, \omega_P
\]
for some constants $C \gg 1$ and $1\gg a >0$.
The hyperplane bundle associated with ${\bf{Proj}}(\bigoplus_{m=0}^{\infty} \mathrm{S}^m \mathcal{E})$
is relatively ample with respect to ${\bf{Proj}}(\bigoplus_{m=0}^{\infty} \mathrm{S}^m \mathcal{E}) \to X$, 
and the structure sheaf $\mathcal{O}_{\mathbb{P}(\mathcal{E})}$
is relatively ample with respect to the normalization $\mathbb{P}(\mathcal{E}) \to {\bf{Proj}}(\bigoplus_{m=0}^{\infty} \mathrm{S}^m \mathcal{E})$.
This shows that the line bundle $\mathcal{O}_{\mathbb{P}(\mathcal{E})}(1)$ is relatively $\pi_{\mathcal{E}}$-ample.
Hence, there exists a smooth metric $g_1$ on $\mathcal{O}_{\mathbb{P}(\mathcal{E})}(1)$ such that
$
\sqrt{-1}\Theta_{g_1} + C\, \pi_{\mathcal{E}}^{*} \omega_X > 0 
$
for some constant $C > 0$.
Let $E$ be an effective $\pi$-exceptional divisor such that $-E$ is relatively $\pi$-ample.
Then, the line bundle $\mathcal{O}_{P}(E)$ admits a smooth metric $g_{2}$ such that  
\begin{align*}
&\sqrt{-1}\Theta_{\pi^{*}g_{1}} + C p^{*} \omega_{X} - \delta \sqrt{-1}\Theta_{g_{2}}(E) \\
= &\pi^{*}\big( \sqrt{-1}\Theta_{g_{1}} + C \pi_{\mathcal{E}}^{*} \omega_{X} \big) - \delta \sqrt{-1}\Theta_{g_{2}}(E) >0
\end{align*}
for $1 \gg \delta > 0$.
We define a singular metric $h$ on $L = (L - \delta E) + \delta E$ by  
\[
h := (\pi^{*}g_{1}) \cdot g_{2}^{-\delta} \cdot g_{[\delta E]},
\]
where $g_{[\delta E]}$ denotes the singular metric associated with the effective divisor $\delta E$.  
Then $h$ satisfies the required properties.

Next, we construct singular metrics $\{g_{\e}\}_{\e > 0}$ on $L|_{P_{0}}$ satisfying the following properties: 
\begin{itemize}
\item[$\bullet$] $\sqrt{-1}\Theta_{g_{\e}} + \e p^{*}\omega_{X} \geq 0$ on $P_{0}$;
\item[$\bullet$] $\{x \in P_{0} \mid \nu(g_{\e}, x) > 0\}$ does not dominate $X_{0}$.
\end{itemize}
By the definition of pseudo-effective sheaves,  
there exist singular Hermitian metrics $\{h_{m}\}_{m \in \mathbb{Z}_{+}}$ on ${\rm{S}}^{[m]} \mathcal{E}$  
such that $\sqrt{-1}\Theta_{h_{m}} \geq -\omega_{X} \otimes \mathrm{id}$.  
Note that $p \colon P_{0} \to X_{0}$ is the projective space bundle
$\bb{P}(\mathcal{E}|_{X_{0}}) \to X_{0}$ of the locally free sheaf $\mathcal{E}|_{X_{0}}$
Hence, the pullbacks $p^{*} h_{m}$ define singular Hermitian metrics on $p^{*} \Sym^{m} \mathcal{E}|_{P_{0}}$  
such that $\sqrt{-1}\Theta_{p^{*} h_{m}} \geq -p^{*} \omega_{X} \otimes \mathrm{id}$ on $P_{0}$.
Let $H_{m}$ be the singular Hermitian metric on $mL|_{P_{0}}$ induced by the natural surjective morphism
\[
p^{*} \Sym^{m} \mathcal{E}|_{P_{0}} \to mL|_{P_{0}}.
\]
By construction, we have $\sqrt{-1}\Theta_{H_{m}} \geq -p^{*} \omega_{X}$,  
and $\{x \in P_{0} \mid \nu(H_{m}, x) > 0\}$ does not dominate $X_{0}$.  
Choosing $m$ such that $\e >1/m$, 
the metrics $H_{m}^{1/m}$ on $L|_{P_{0}}$ give the desired family $\{g_{\e}\}_{\e > 0}$.

Finally, we define singular metrics on $L|_{P_{0}}$ by  
\[
h_{\e} := g_{\e/2}^{1 - \delta} \cdot h^{\delta}.
\]
From the curvature properties of $h$ and $g_{\e}$, we obtain
\[
\sqrt{-1}\Theta_{h_{\e}} \geq -\left( \frac{\e}{2} + \delta C \right) p^{*} \omega_{X} + \delta a \omega_{P}.
\]
Setting $\delta := \e / 2C$, we conclude that $h_{\e}$ satisfies the required properties.
\end{proof}

\begin{remark}\label{rem-psef}
In general, a metric $h_{\e}$ on $L|_{P_{0}}$ \textit{cannot} 
be extended to a singular Hermitian metric on $L$. 
For example, let $\mathcal{E}$ be the ideal sheaf of a submanifold $Z \subset X$ of codimension $\ge 2$.  
Then $\mathcal{E}$ is pseudo-effective in the sense of Definition \ref{def-psef}. 
We may take $P$ to be the blow-up of $X$ along $Z$ with exceptional divisor $E$, so that $L = -E$.  
Then, the line bundle $L = -E$ does not admit singular Hermitian metrics $h_{\e}$ satisfying 
$\sqrt{-1}\Theta_{h_{\e}} \geq - \e \omega_{P}$ for all $\e>0$. 
Nevertheless, since we may take $\Lambda:=E$, 
the corresponding metrics on $L+\Lambda=\mathcal{O}_{P}$ can extend. 
\end{remark} 

As noted above, extending a metric $h_{\varepsilon}$ defined only on $P_{0}$ to $P$ is a delicate problem, 
since $P \setminus P_{0}$ can be divisorial.  
Nevertheless, by considering $L + \Lambda$ instead of $L$,  
we can extend $h_{\e}$ to a metric on $L + \Lambda$  
using the identification $L|_{P_{0}} \cong (L + \Lambda)|_{P_{0}}$,  
which is a key step in the proof of Theorem \ref{thm-psef}.

\begin{thm}\label{thm-extension}
Let $h$ be a singular metric on $L|_{P_{0}}$ such that 
\begin{align*}
\sqrt{-1}\Theta_{h} + C\, p^{*}\omega_{X} \geq a \omega_{P} \quad \text{on } P_{0},
\end{align*}
for some constants $C \gg 1$ and $0<a\ll 1$.
Then, when $h$ is regarded as a metric on $(L+\Lambda)|_{P_{0}}$ via the identification 
$L|_{P_{0}} \cong (L+\Lambda)|_{P_{0}}$ $($for which we use the same notation $h$$)$, 
the weight of $h$ is locally bounded above near $P\setminus P_{0}$. 
In particular, the metric $h$ extends across $P\setminus P_{0}$ to a singular metric on $L+\Lambda$ 
$($which we denote by the same symbol $h$$)$. 

More precisely, this means that the weight of 
the metric $h\cdot h_{\Lambda}$ on $(L+\Lambda)|_{P_{0}}$ is locally bounded above, 
and hence extends to a singular metric on $L+\Lambda$,
where $h_{\Lambda}$ is the singular metric on $\Lambda$  associated with the effective divisor $\Lambda$.
\end{thm}

As a direct consequence, 
the metric $h_{\e}$ constructed in Lemma \ref{lemm-metric} 
extends to a singular metric on $L+\Lambda$ satisfying the desired properties as in Theorem \ref{thm-psef}, 
completing the proof of Theorem \ref{thm-psef}. 

\begin{proof}
We regard the metric $h$ on $L|_{P_{0}}$ as a metric on $(L+\Lambda)|_{P_{0}}$ via the identification $L|_{P_{0}} \cong (L+\Lambda)|_{P_{0}}$.  
It suffices to show that the local weight $\varphi$ of $h$ is bounded from above in a neighborhood of $P\setminus P_{0}$.  
The problem is local on $X$, so we may assume that $X$ is Stein.  
Then, after replacing $h$ by $h\,e^{-p^{*}\psi}$, where $\psi$ is a local potential of $C \omega_{X}$, 
we may assume that $\sqrt{-1}\Theta_{h}$ is a K\"ahler current by the curvature assumption.  

For $m\in\mathbb{Z}_{>0}$, we consider the $L^{2}$-space  
\[
H^{0}_{L^{2}}(P_{0},m(L+\Lambda))_{h^{m},\omega_{P}}
\]
of holomorphic sections $\sigma$ on $P_{0}$ with finite $L^{2}$-norm $\|\sigma\|_{h^{m},\omega_{P}}< \infty$.  
Let $B_{m}$ denote the Bergman kernel metric of this Hilbert space,  
whose local weight $b_{m}$ is defined by 
\[
b_{m}(p)
          :=\frac{1}{m}\sup\Bigl\{\log|\sigma|(p) \;\Big|\;\|\sigma\|_{h^{m},\omega_{P}}\le 1\Bigr\} \quad \text{ for } p\in P_{0}. 
\]
Then, for every relatively compact set $K\Subset X$, there exists a constant $C=C(K)$, independent of $m$, such that
\begin{equation}\label{eq-dem}
\varphi(p)\le b_{m}(p)+\frac{C}{m}\qquad\text{for }p\in P_{0}\cap p^{-1}(K).
\end{equation}
This can be verified by the standard $L^{2}$-methods for $\dbar$-equations 
(see, for example, \cite[(13.21) Theorem, p.~138]{Dem10}). 
We briefly explain the proof for the reader's convenience.

Fix a point $p_{0} \in P_{0}\cap p^{-1}(K)$. 
Take an open ball $B \Subset P$ centered at $p_{0}$ and 
a hypersurface $H\subset X$ such that $X\setminus X_{0} \subset H$ and $p_{0} \not \in p^{-1}(H)$.  
Note that $P\setminus p^{-1}(H)$ is a weakly pseudoconvex K\"ahler manifold and that $B\setminus p^{-1}(H)$ is Stein. 
This allows us to solve $\dbar$-equations on $P\setminus p^{-1}(H)$ and 
to apply the $L^{2}$-extension theorem on $B\setminus  p^{-1}(H)$.
If $\varphi(p_{0})=-\infty$, the inequality is obvious, so we may assume $\varphi(p_{0})>-\infty$.  
For a given $\alpha\in\mathbb{C}$, 
by the $L^{2}$-extension theorem, there exists a holomorphic function $f$ on $B\setminus p^{-1}(H)$ such that
\[
f(p_{0})=\alpha \quad \text{ and } \quad 
\int_{B\setminus p^{-1}(H)}|f|^{2}e^{-m\varphi}\,dV_{\omega_{P}}\le C_{1}|\alpha|^{2}e^{-m\varphi(p_{0})}.
\]
Fix a relatively compact subset $K_{P}\Subset B$ and consider a cut-off function $\theta$ supported in $B$ and equal to~$1$ on $K_{P}$.  
Since $\sqrt{-1}\Theta_{h}$ is a K\"ahler current, 
the standard $L^{2}$-methods for $\dbar$-equations enable us to solve
$\overline{\partial}g=\overline{\partial}(\theta f)$ on $P\setminus p^{-1}(H)$ with the $L^{2}$-estimate
\[
\int_{P\setminus p^{-1}(H)}|g|_{h^{m}}^{2}\,e^{-2n\theta\log|z-z(p_{0})|}\,dV_{\omega_{P}}
\le C_{2}\int_{B\setminus p^{-1}(H)}|f|^{2}e^{-m\varphi}\,dV_{\omega_{P}}, 
\]
where $z$ is a coordinate on $B$ and $n:=\dim P$. 
Then, the section $\sigma:=\theta f -g$ satisfies
\begin{align}\label{eq-norm}
\sigma(p_{0})= f(p_{0})=\alpha \quad \text{ and } \quad  
\|\sigma\|_{h^{m}}\le C_{3}|\alpha|^{2}e^{-m\varphi(p_{0})}. 
\end{align}
Here, on the left-hand side, we identify $\sigma$ with a holomorphic function via a fixed local frame.
The section $\sigma$ is defined on $P\setminus p^{-1}(H)$. 
On the other hand, since $h$ is defined even near every point of $p^{-1}(H)\cap P_{0}$, 
every $L^{2}$-section on $P\setminus p^{-1}(H)$ extends to $P_{0}$, giving the isomorphism
\[
H^{0}_{L^{2}}(P_{0},m(L+\Lambda))_{h^{m},\omega_{P}}
\cong H^{0}_{L^{2}}(P\setminus p^{-1}(H),m(L+\Lambda))_{h^{m},\omega_{P}}.
\]
This shows that $\sigma$ can be regarded as a section on $P_{0}$. 
We can choose $\alpha$ so that the right-hand side in \eqref{eq-norm} equals $1$, and thus we obtain inequality \eqref{eq-dem}.

\smallskip 
Henceforth, to simplify notation, we assume $m = 1$ and set $L_{\Lambda} := L + \Lambda$.
We will show that the weight $b:=b_{1}$ of the Bergman kernel is locally bounded from above near $P \setminus P_{0}$.  
The crucial point is that every section of $L_{\Lambda}$ on $P_{0}$ extends holomorphically to $P$, 
since $p_{*}(L_{\Lambda})$ is reflexive by the choice of $\Lambda$. 

Fix a trivialization of $L_{\Lambda}$ on an open ball $B \subset P$.  
Then, for each point $p \in B_{0}:=B \cap P_{0}$, we consider the evaluation map
\[
\mathrm{ev}_{p} \colon 
H^{0}_{L^{2}}(P_{0},L_{\Lambda})_{h,\omega_{P}}
\longrightarrow
H^{0}(B_{0},L_{\Lambda}) \cong H^{0}(B_{0},\mathcal{O}_{B})
\longrightarrow
\mathbb{C}. 
\]
By taking duals and the operator norm, 
we can interpret the weight $b$ as the composition 
\[
b=b_{1}\colon\,
B_{0}\xrightarrow{\ \mathrm{ev}^{*}\ }
\bigl(H^{0}_{L^{2}}(P_{0},L_{\Lambda})_{h,\omega_{P}}\bigr)^{*}
\xrightarrow{\ \|\cdot\|_{\mathrm{op}}\ }
\mathbb{R}, 
\]
where the dual space is equipped with the strong topology.
By Montel's theorem, the map $\mathrm{ev}^{*}$ is continuous.  
Hence, it is enough to extend $\mathrm{ev}^{*}$ continuously to $B$; 
once this is done, the weight function $b=b_{1}$ is bounded on compact subsets of $B$.

The restriction map
\[
H^{0}(P,L_{\Lambda}) \longrightarrow H^{0}(P_{0},L_{\Lambda})
\]
is an isomorphism of Fr\'echet spaces with the topology induced by local sup-norms. 
Indeed, this map is clearly continuous, injective, and linear. 
Moreover, it is surjective, since we have
\[
H^{0}(P,L_{\Lambda})
\cong H^{0}(X,p_{*}(L_{\Lambda}))
\cong H^{0}(X_{0},p_{*}(L_{\Lambda}))
\cong H^{0}(P_{0},L_{\Lambda})
\]
by reflexivity. 
Hence, the map is a homeomorphism by the open mapping theorem.
This induces a continuous map
\[
H^{0}(P,L_{\Lambda})^{*}
\cong H^{0}(P_{0},L_{\Lambda})^{*}
\to
\bigl(H^{0}_{L^{2}}(P_{0},L_{\Lambda})_{h,\omega_{P}}\bigr)^{*}. 
\]
On the other hand, there exists a continuous map
\[
B \to H^{0}(B, L_{\Lambda})^{*} \to H^{0}(P,L_{\Lambda})^{*}.
\]
The composition of these maps extends $\mathrm{ev}^{*}$ continuously to $B$, 
proving the desired local boundedness.
\end{proof}

The following theorem is an almost direct consequence of the above arguments.  
For completeness, we briefly sketch the proof.

\begin{thm}\label{thm-ext}
Let $G$ be a singular Hermitian metric on $\Sym^{m}\mathcal{E}$
such that 
$$\sqrt{-1}\Theta_{G} \geq \theta \otimes \id, $$
where $\theta$ is a smooth $(1,1)$-form on $X$ which locally admits potentials and
$m \in \mathbb{Z}_{+}$ is a fixed integer.
Consider the singular metric $g$ on $mL|_{P_{0}}$
induced by $G$ and the natural surjective morphism
\[
p^{*}\Sym^{m}\mathcal{E}|_{P_{0}} \to mL|_{P_{0}}.
\]
Then, after replacing $\Lambda$ by a sufficiently large multiple, 
the metric $g$ extends to a singular metric on $m(L+\Lambda)$ 
$($which we denote by the same symbol $g$$)$
via the identification $L|_{P_{0}} \cong (L+\Lambda)|_{P_{0}}$, 
and the extension satisfies $\sqrt{-1}\Theta_{g} \geq p^{*}\theta$.
\end{thm}

\begin{proof}
As explained in the proof of Lemma \ref{lemm-metric}, 
there exists a singular metric $h$ on $L$ (not only on $L|_{P_{0}}$)
such that $\sqrt{-1}\Theta_{h}=[\delta E]$ modulo smooth forms and
\[
\sqrt{-1}\Theta_{h} + C\, p^{*}\omega_{X} \geq a\, \omega_{P}
\]
for some constants $C \gg 1$ and $1\gg a>0$, 
where $E$ is an effective $\pi$-exceptional divisor. 
Then, the metric $g \cdot h^{\delta}$ on $P_{0}$
satisfies the assumptions needed to apply the argument of Theorem\ref{thm-extension}. 
Thus, its weight is locally bounded above on $P$ when $g \cdot h^{\delta}$ is regarded as a metric on $(L+\Lambda)|_{P_{0}}$. 
Precisely, this means that $g \cdot h^{\delta} \cdot h_{\Lambda}$ extends to a metric on $(L+\Lambda)$. 
Hence $g \cdot h_{2\Lambda}$ also extends,
since $\Supp(E)\subset \Supp(\Lambda)$ and $\sqrt{-1}\Theta_{h}=[\delta E]$ modulo smooth forms.
\end{proof}

We now give a characterization of the condition $c_{1}(\mathcal{E})=0$.
Although this characterization may look obvious at first glance, the proof relies on Theorem~\ref{thm-ext} and \cite[Theorem 1.1]{CGNPPW},
and hence involves a subtle point.

\begin{proposition}\label{prop-flat}
Let $\mathcal{E}$ be a torsion-free sheaf of rank $1$ on $X$.
Let $p \colon P \to X$ be a modification such that $M:=(p^{*}\mathcal{E})/\tor$ is an invertible sheaf.
Then, the following conditions are equivalent:
\begin{itemize}
\item[$(1)$] $c_{1}(M)$ is represented by a $($not necessarily effective$)$ $p$-exceptional divisor.

\item[$(2)$] $\mathcal{E}$ admits a singular  metric $h$
such that $\sqrt{-1}\Theta_{h}=0$ holds on the Zariski open set $X_{0}=X_{\reg} \cap X_{\mathcal{E}}$.
\end{itemize}
If, in addition, the sheaf $\mathcal{E}$ is pseudo-effective, then the above conditions are equivalent to the following condition:
\begin{itemize}
\item[$(3)$] $\int_{X} c_1 (\mathcal{E}) \wedge \omega^{n-1}=0$ holds for some K\"ahler form $\omega$ on $X$.
\end{itemize}
\end{proposition}

\begin{proof}
We apply Setup \ref{setup} in the case where $\rk \mathcal{E}=1$.
We may assume that the morphism $p \colon P \to X$ in \eqref{eq-comm} 
is a modification such that it is an isomorphism over $X_{0}$
and $M:=(p^{*}\mathcal{E})/\tor$ is an invertible sheaf.
Since $p$ is an isomorphism over $X_{0}$, there exists a $($not necessarily effective$)$ $p$-exceptional divisor $F$ such that
\[
M \cong L+F.
\]
Indeed, the natural morphisms
\[
p^{*} p_{*}\bigl(\mathcal{O}_{P}(M)\bigr) \to p^{*} p_{[*]}\bigl(\mathcal{O}_{P}(M)\bigr)
\quad\text{and}\quad
p^{*} p_{*}\bigl(\mathcal{O}_{P}(M)\bigr) \to \mathcal{O}_{P}(M)
\]
are isomorphisms on $P_{0}$, and hence both their kernels and cokernels are torsion sheaves supported on a $p$-exceptional divisor.
This implies that the difference $p^{[*]} p_{[*]}\bigl(\mathcal{O}_{P}(M)\bigr) - M$ is represented by a $p$-exceptional divisor.
The same argument applies to $\mathcal{O}_{P}(L)$.
The claimed isomorphism follows from
\[
p_{[*]}\bigl(\mathcal{O}_{P}(M)\bigr) \cong \mathcal{E}^{**} \cong
p_{[*]}\bigl(\mathcal{O}_{P}(L)\bigr).
\]

Assume (2).
Then $L|_{P_{0}}$ admits a smooth metric $h$ with $\sqrt{-1}\Theta_{h}=0$.
By Theorem~\ref{thm-ext}, this metric on $L|_{P_{0}}$ extends to a possibly singular metric on $L+\Lambda$
via $L|_{P_{0}} \cong (L+\Lambda)|_{P_{0}}$ $($which we denote by the same symbol $h$$)$.
Moreover, the curvature $\sqrt{-1}\Theta_{h}$ is supported on the $p$-exceptional locus.
Therefore, the first Chern class of $L+\Lambda$, and hence that of $M$, is represented by a $p$-exceptional divisor, which yields (1).

Conversely, assume (1).
Since the support of $\Lambda$ coincides with the $p$-exceptional locus,
we see that $c_{1}(L+m \Lambda)$ is represented by an effective $p$-exceptional divisor $E$ for $m\gg1$.
Therefore, there exists a singular metric $g$ on $L+m \Lambda$ whose curvature current is the integration current $[E]$.
The metric $h:=p_{*}g$ on $\mathcal{E}|_{X_{0}}$ defined by pushforward satisfies the desired property.

Clearly (1) implies (3).
Assume (3).
Since $\mathcal{E}$ is pseudo-effective,
the above argument shows that $L+m\Lambda$
admits a singular metric $h$ whose curvature current is semipositive for $m\gg1$.
Then, we see from (3) that $\sqrt{-1}\Theta_{h}=0$ on $P_{0}$, which yields (1).
\end{proof}

\subsection{Proof of Theorem \ref{thm-stable}}\label{subsec-proof2}
We begin the proof of Theorem~\ref{thm-stable}.

\begin{proof}[Proof of Theorem \ref{thm-stable}]
We assume that $\mathcal{E}$ is pseudo-effective and has vanishing first Chern class.
We first show that $\mathcal{E}$ is weakly positively curved in the following sense: 

\begin{lemma}\label{lemma-weak}
The sheaf $\mathcal{E}$ admits singular metrics $\{H_{\e}\}$
such that $\sqrt{-1}\Theta_{H_{\e}} \geq - \e \omega_{X} \otimes \id$ and 
each $H_{\e}$ is smooth over a non-empty Zariski open subset of $X$.
Moreover, let $\{h_{\e}\}$ be the singular metrics on $L+\Lambda$ obtained by extending
the induced metrics on $L|_{P_{0}}$ as in Theorem \ref{thm-ext}.
Then $\{h_{\e}\}$ satisfy the following properties:
\begin{itemize}
\item[$(1)$] $\sqrt{-1}\Theta_{h_{\e}} + \e p^{*}\omega_{X} \geq \delta_{\e}\,\omega_{P}$ holds on $P$
for some $\delta_{\e}>0$;
\item[$(2)$] $h_{\e}$ is induced by $H_{\e}$ and the surjective morphism
$p^{*}\mathcal{E}|_{P_{0}} \to L|_{P_{0}}$;
\item[$(3)$] $h_{\e}$ is smooth over a non-empty Zariski open subset of $P$.
\end{itemize}
\end{lemma}

\begin{proof}
In the notation of Setup \ref{setup}, by \cite[Lemmas 2.3 and 2.4]{MWWZ25},
there exist singular  metrics $\{h_{\e}\}$ on $L+\Lambda$ such that
\begin{itemize}
\item[(a)] $\sqrt{-1}\Theta_{h_{\e}} + \e p^{*}\omega_{X} \geq \delta_{\e}\,\omega_{P}$ holds on $P$
for some $\delta_{\e}>0$;
\item[(b)] $\{x \in P\mid \nu(h_{\e}, x) > 0\}$ is not dominant over $X_{0}$.
\end{itemize}
By Demailly's approximation theorem \cite{Dem92},
we can replace $\{h_{\e}\}$ with singular metrics with analytic singularities so that $h_{\e}$ is smooth over a Zariski open set in $X$ by property $(b)$.

Consider the natural inclusion
$$
p_{*} \bigl( \mathcal{O}_{P}(K_{P/X} + (r+1) (L+\Lambda )) \otimes \I{h_{\e}^{r+1}} \bigr)
\subset
p_{*} \bigl( \mathcal{O}_{P}(K_{P/X} + (r+1) (L+\Lambda )) \bigr).
$$
By (a), the positivity of direct image sheaves \cite{PT18, HPS18}
shows that the left-hand side is weakly positively curved.
More precisely, the induced $L^{2}$-metrics $\{G_{\e}\}$ on the left-hand side satisfy
$\sqrt{-1}\Theta_{G_{\e}} \geq - \e \omega_{X} \otimes \id$.
Moreover, by construction, $G_{\e}$ is smooth over a non-empty Zariski open subset of $X$.
The above inclusion is generically an isomorphism, since the multiplier ideal sheaf $\I{h_{\e}^{r+1}}$ is trivial over a Zariski open subset of $P$.
Hence, the right-hand side also admits singular Hermitian metrics $\{H_{\e}\}$ with the same properties (see Proposition \ref{genersurjet}).

On the other hand, since $p \colon P_{0} \to X_{0}$ is the projective space bundle associated with $\mathcal{E}|_{X_{0}}$,
we have
$$
K_{P/X} + (r+1) (L+\Lambda) \cong L+ p^{*} \det \mathcal{E} \quad \text{ over }X_{0},
$$
where $r$ is the rank of $\mathcal{E}$.
Here we used $\Lambda=0$ over $X_{0}$.
This shows that 
\begin{align*}
p_{*} \bigl( \mathcal{O}_{P}(K_{P/X} + (r+1) (L+\Lambda )) \bigr)  \cong
p_{*} \big( \mathcal{O}_{P_{0}}(L+ p^{*} \det \mathcal{E}) \big)=(\mathcal{E} \otimes \det \mathcal{E})
\end{align*}
over $X_{0}$.
Note that this identification holds only over $X_{0}$ 
since $\det \mathcal{E}$ is not necessarily invertible on $X$.

By the assumption $c_{1}(\mathcal{E})=0$,
Proposition \ref{prop-flat} shows that there exists a smooth  metric $g$ on $\det \mathcal{E}|_{X_{0}}$ such that $\sqrt{-1}\Theta_{g}=0$.
Set $H_{\e}:=G_{\e} \cdot g^{-1}$.
Since $\codim (X \setminus X_{0}) \geq 2$, the metric $H_{\e}$ extends across $X\setminus X_{0}$ as a singular Hermitian metric on $\mathcal{E}$,
and it satisfies the desired properties.

Finally, the remaining assertions follow from the construction of $h_{\e}$ and Theorem \ref{thm-ext}.
Indeed, the metric $h_{\e}$ on $L|_{P_{0}}$ is induced by $H_{\e}$ and the surjective morphism
$$
p^{*}\mathcal{E}|_{P_{0}} \to L|_{P_{0}}.
$$
By Theorem \ref{thm-ext}, it extends to a singular  metric $h_{\e}$ on $L+\Lambda$.
This completes the proof.
\end{proof}

Let $h$ be the singular metric on $L+\Lambda$ obtained as a limit of $\{h_{\e}\}$.
We consider the upper level set of Lelong numbers of $h$:
\[
P_{+}:=\{x \in P \mid \nu(h, x) > 0\}.
\]
We divide the argument into two cases:
\begin{itemize}[leftmargin=5.5em]
\item[{\bf Case 1}:] $P_{+}$ is not dominant over $X$.
\item[{\bf Case 2}:] $P_{+}$ is dominant over $X$.
\end{itemize}

\begin{case}\label{case1}
Suppose that $P_{+}$ is not dominant over $X$.
We will show that $\mathcal{E}^{**}|_{X_{\reg}}$ is Hermitian flat.
By applying the positivity of direct image sheaves in \cite{HPS18} (see also \cite{PT18}) once again,
we conclude that the direct image sheaf
$$
p_{*}\bigl(\mathcal{O}_{P}(K_{P/X} + (r+1)(L+\Lambda)) \otimes \I{h^{r+1}}\bigr)
\subset
p_{*}\bigl(\mathcal{O}_{P}(K_{P/X} + (r+1)(L+\Lambda))\bigr)
$$
admits a positively curved singular Hermitian metric.
Under the assumption of Case~\ref{case1}, the above inclusion is generically an isomorphism.
Hence, by Proposition \ref{genersurjet}, we conclude that
$$
p_{*}\bigl(\mathcal{O}_{P}(K_{P/X} + (r+1)(L+\Lambda))\bigr)
$$
also admits a positively curved singular Hermitian metric.
Arguing as in Lemma~\ref{lemma-weak},
and using the assumption $c_{1}(\mathcal{E})=0$,
we obtain a positively curved singular Hermitian metric $H$ on $\mathcal{E}$.

The determinant metric $\det H$ induces a positively curved singular metric on $\det\mathcal{E}$.
By $c_{1}(\mathcal{E})=0$, the weight of $\det H$ is pluriharmonic (hence smooth) on $X_{0}$
(see the proof of Proposition \ref{prop-flat}).
By \cite[Theorem~1.6]{Rau15}, the curvature of $H$ is well-defined. Then $H$ is Hermitian flat on $\mathcal{E}|_{X_{0}}$ by \cite[Cororally~2.9]{CP17} or \cite[Lemma~3.6]{HIM22}.
This shows that $\mathcal{E}|_{X_{0}}$ is induced by a unitary representation of $\pi_{1}(X_{0})$.

Since $X_{\reg}$ is smooth and $\codim(X \setminus X_{\mathcal{E}})\geq 2$,
the inclusion $X_{0}\hookrightarrow X_{\reg}$ induces an isomorphism
\[
\pi_{1}(X_{0}) \xrightarrow{\ \sim\ } \pi_{1}(X_{\reg})
\]
by the van Kampen theorem.
Thus, the unitary representation extends to $\pi_{1}(X_{\reg})$, and hence $\mathcal{E}^{**}$ is Hermitian flat on $X_{\reg}$.
\end{case}

\begin{case}\label{case2}
Suppose that $P_{+}$ is dominant over $X$.
Assume that $\mathcal{E}$ is stable with respect to a K\"ahler form $\omega_{X}$.
We will show that this case cannot occur.

\medskip
\noindent\emph{The smooth case.}
To illustrate the idea of the proof, we first consider the case where
$X$ is smooth and $\mathcal{E}$ is locally free.

Since $\mathcal{E}$ is stable and $X$ is smooth, by \cite{UY86}, there exists a smooth Hermitian--Einstein metric $G$ on $\mathcal{E}$, i.e.\
\[
\sqrt{-1}\Theta_{G}(\cE)\wedge \omega_{X}^{n-1}\equiv 0
\quad\text{on } X.
\]
Let $g$ be the induced metric on $L :=\mathcal{O}_{\mathbb P (\cE)} (1)$.
Set $P := \mathbb P (\cE)$, and let $p\colon P \to X$ be the natural projection.

We now construct a destabilizing subsheaf.
By Proposition \ref{addlemma} and Lemma \ref{lemma-weak}~(2), and arguing as in the proof of Lemma~\ref{lemma-weak},
there exists $0<\delta<1$ such that
\[
\cG := p_* \Big(
\mathcal{O}_{P}\bigl(K_{P/X} + (r+1)L \bigr)\otimes
\mathcal{I}\bigl((h^{\delta}\cdot g^{(1-\delta)})^{r+1}\bigr)
\Big)
\]
is a nonzero subsheaf of $\mathcal{E}\otimes \det\mathcal{E}$, and the inclusion is proper.
By the construction of $g$, we have
\[
\sqrt{-1}\Theta_{g}(L)\wedge p^*\omega_{X}^{n-1}\geq 0
\quad\text{over } X.
\]
Moreover, the same holds for $h$ since $\sqrt{-1}\Theta_{h}(L)\geq 0$ in the sense of currents.
Hence we have
\[
\sqrt{-1}\Theta_{h^{\delta}\cdot g^{(1-\delta)}}(L)\wedge p^*\omega_{X}^{n-1}\geq 0
\quad\text{over } X.
\]
Therefore, applying \cite[Theorem 1.6]{CCP19}, we obtain
\[
\sqrt{-1}\Theta_{\det H}(\det \cG)\wedge \omega_{X}^{n-1}\geq 0
\quad\text{on } X,
\]
where $H$ is the $L^{2}$-metric on $\cG$ induced by $h^{\delta}\cdot g^{(1-\delta)}$.

By the above curvature inequality, we have
\[
\int_X c_1(\cG)\wedge \omega_X^{n-1}\ge 0.
\]
By $c_1(\cE)=0$, Proposition \ref{prop-flat} implies
\[
\int_X c_1\bigl(\cG\otimes(\det\cE)^{*}\bigr)\wedge \omega_X^{n-1}
=
\int_X c_1(\cG)\wedge \omega_X^{n-1}\ge 0.
\]
By construction, the sheaf $\cG\otimes(\det\cE)^{*}$ is a nonzero subsheaf of $\cE$, and the inclusion is proper.
Therefore, the sheaf $\cG\otimes(\det\cE)^{*}$ destabilizes $\cE$, a contradiction.

\medskip
\noindent\emph{The singular case.}
We now treat the general case.
Let $\pi\colon \tilde{X}\to X$ be a resolution of singularities of $X$ such that
$$E:=\pi^{*}\cE/\tor$$ is a locally free sheaf on $\tilde{X}$.
Let $\omega_{\tilde{X}}$ be a K\"ahler metric on $\tilde{X}$.
By \cite[Section 3]{CGNPPW} or \cite{Tom21}, the locally free sheaf $E$ is stable with respect to $\pi^*\omega_X +\ep \omega_{\tilde{X}}$ for $0< \ep\ll 1$.
Hence there exists a smooth Hermitian--Einstein metric $G_\ep$ on $E$ with respect to $\pi^*\omega_X +\ep \omega_{\tilde{X}}$, i.e.\ it satisfies
$$\sqrt{-1} \Theta_{G_\ep} (E) \wedge (\pi^*\omega_X +\ep \omega_{\tilde{X}})^{n-1}
= c_\ep \id_E \cdot  (\pi^*\omega_X +\ep \omega_{\tilde{X}})^n. $$
Here $c_\ep$ is the constant satisfying $\lim_{\ep \to 0} c_\ep =0$.
Let $g_\ep$ be the quotient metric on $\mathcal{O}_{\mathbb P (E)} (1)$ induced by $G_\ep$.
The metric $g_\ep$ satisfies
\begin{equation}\label{curvc}
\sqrt{-1}\Theta_{g_\ep }(\mathcal{O}_{\mathbb P (E)} (1))\wedge \big(\tilde{p}^* (\pi^*\omega_X +\ep \omega_{\tilde{X}})\big)^{n-1}\geq c_\ep  \big(\tilde{p}^*(\pi^*\omega_X +\ep \omega_{\tilde{X}})\big)^n  \quad\text{over }  \tilde{X},
\end{equation}
where $\tilde{p}\colon \mathbb P (E) \to \tilde{X}$ is the natural projection.

Let $\tau\colon \tilde{P}\to P$ be a bimeromorphic morphism such that $\tilde{P}$ is smooth and
it satisfies the following commutative diagram:
\[
\xymatrix{
\tilde{P} \ar[d]_{\tau} \ar[rrd]^{\tilde{\pi}} &  &   \\
\mathbb P (E)\ar[d]_{\tilde{p}}  &  & P \ar[d]_{p}  \\
\tilde{X} \ar[rr]_{\pi}&  & X.
}
\]
Then, we can write
\begin{equation}\label{addrel}
\tilde{L} :=\tilde{\pi}^{*}(L+\Lambda) + \Lambda_1 \cong \tau^* \mathcal{O}_{\mathbb P (E)} (1) + \Lambda_2
\end{equation}
for some effective divisors $\Lambda_1$ and $\Lambda_2$ supported in the inverse image of $(X\setminus X_0)$
(see the proof of Proposition \ref{prop-flat}).

Let $\tilde{h}$ be the singular metric on $\tilde{L}$ induced by the metric $h$ on $L$ and the metric
defined by the natural section of $\Lambda_1$.
By construction, the curvature $\sqrt{-1}\Theta_{\tilde{h}}$ is semipositive.
On the other hand, thanks to \eqref{addrel}, we obtain the singular metric $\tilde{g}_\ep$ on $\tilde{L}$
induced by the metric $g_\ep$ on $\mathcal{O}_{\mathbb P (E)} (1)$ and the metric defined by the natural section of $\Lambda_2$.

Consider the direct image sheaf
\[
\tilde{\cG} := (\tilde{p} \circ \tau )_*\Big(
\mathcal{O}_{\tilde{P}}\bigl(K_{\tilde{P}/\tilde{X}}+(r+1)\tilde{L}\bigr)\otimes
\mathcal{I}\bigl((\tilde{h}^{\delta}\cdot \tilde{g}_\ep^{(1-\delta)})^{r+1}\bigr)
\Big)
\]
and the induced $L^{2}$-metric $\tilde{H}_\ep$ on $\tilde{\cG}$.
Note that $g_{\ep}$ is smooth, hence the multiplier ideal sheaf
$\mathcal{I}\bigl((\tilde{h}^{\delta}\cdot \tilde{g}_\ep^{(1-\delta)})^{r+1}\bigr)$ depends only on $\delta$.
In particular, the sheaf $\tilde{\cG}$ is independent of $\ep$.

Applying \cite[Theorem 1.6]{CCP19} and using \eqref{curvc}, we obtain
\[
\sqrt{-1}\Theta_{\det \tilde{H}_\ep}(\det \tilde{\cG})\wedge (\pi^*\omega_X +\ep \omega_{\tilde{X}})^{n-1}\geq c_\ep (\pi^*\omega_X +\ep \omega_{\tilde{X}})^{n}
\quad\text{on } \tilde{X}.
\]
Taking $\ep \to 0$ and using $\lim_{\ep \to 0}c_\ep=0$, we obtain
\begin{equation}\label{strictsheaf}
\int_{\tilde{X}} c_{1}(\tilde{\cG})\wedge \pi^{*}\omega_{X}^{n-1}\geq 0.
\end{equation}

Since $\pi_*\tilde{\cG} \subset \cE\otimes \det \cE$ on $X_{0}$, we have
\[
\bigl(\pi_{[*]}\tilde{\cG}\bigr) \subset \cE[\otimes ]\det\cE \quad\text{on } X.
\]
By Proposition \ref{addlemma} and Lemma \ref{lemma-weak}~(2), there is $\delta>0$
such that the left-hand side is a nonzero proper subsheaf of $\cE[\otimes]\det\cE$.
 
Moreover, the slope $\mu_{\omega_X}\bigl(\pi_{[*]}\tilde{\cG}\bigr)$ can be computed as
\[
\int_{\tilde{X}} c_{1}\Bigl(\pi^{*}\bigl(\pi_{[*]}\tilde{\cG}\bigr)/\tor\Bigr)\wedge \pi^{*}\omega_{X}^{n-1}
=
\int_{\tilde{X}} c_{1}(\tilde{\cG})\wedge \pi^{*}\omega_{X}^{n-1}.
\]
Here the equality follows since $\int_D \pi^{*}\omega_{X}^{n-1}=0$ for any $\pi$-exceptional divisor $D$, and
\[
c_1(\tilde{\cG})-c_1\Bigl(\pi^{*}\bigl(\pi_{[*]}\tilde{\cG}\bigr)/\tor\Bigr)
\]
is supported on the $\pi$-exceptional locus (cf.\ the argument in the proof of Proposition \ref{prop-flat}).
Therefore, by \eqref{strictsheaf}, we obtain
\begin{equation}
 \mu_{\omega_X}\bigl(\pi_{[*]}\tilde{\cG}\bigr)\geq 0.
\end{equation}
This contradicts the stability of $\cE$ with respect to $\omega_{X}$.

\end{case}
\end{proof}

\subsection{Proof of Theorem \ref{thm-main}}\label{subsec-proof1}
In this subsection, we prove Theorem \ref{thm-main}.
\begin{proof}[Proof of Theorem \ref{thm-main}~(i)]
By Theorem \ref{thm-maxqetale}, there exists a maximally quasi-\'etale cover
$$
\nu \colon X' \to X. 
$$
We will prove assertion~(i) for this maximally quasi-\'etale  cover $\nu \colon X' \to X$.
The proof is divided into two steps.

\setcounter{step}{0}
\begin{step}
The reflexive pullback $\nu^{[*]}\mathcal{E}$ is also pseudo-effective and has vanishing first Chern class.
Fix a K\"ahler form $\omega_{X'}$ on $X'$.
By Proposition \ref{semistable}, the sheaf $\nu^{[*]}\mathcal{E}$ is semistable with respect to $\omega_{X'}$.
Hence, we can consider the Jordan--H\"older filtration of $\nu^{[*]}\mathcal{E}$:
\begin{align}\label{eq-jh}
0 =: \cE_0 \subset \cE_1 \subset \cE_2 \subset \cdots \subset \cE_k := \nu^{[*]}\mathcal{E}.
\end{align}
In this step, we prove that each $(\cE_i/\cE_{i-1})^{**}$ is a Hermitian flat locally free sheaf on $X'$.

By construction, the quotient $\cE_i/\cE_{i-1}$ is stable, and its slope vanishes. 
The quotient $\cE_i/\cE_{i-1}$ is torsion-free, hence locally free in codimension $1$,
so we obtain
\[
c_{1}\bigl((\cE_i/\cE_{i-1})^{**}\bigr)=c_{1}(\cE_i/\cE_{i-1}) =0
\]
by the vanishing of the slope and pseudo-effectivity.

We first consider the inclusion $\cE_1 \subset \cE_k=\nu^{[*]}\mathcal{E}$, which induces the generically surjective morphism
$\cE_k^{*}\to \cE_1^{*}$.
By Propositions \ref{genersurjet} and \ref{dualpsef}, the sheaf $\cE_1^{*}$ is pseudo-effective.
Applying Theorem \ref{thm-stable} to $\cE_1^{*}$, 
we conclude that $\cE_1^{*}$ is a Hermitian flat locally free sheaf on $X'_{\reg}$ by $c_1(\cE_1^{*})=0$. 
By Theorem \ref{thm-maxqetale}~(2), it follows that $\cE_1^{*}$ is a Hermitian flat locally free sheaf on $X'$.

The quotient $\cE_k/\cE_1$ is also pseudo-effective and has vanishing first Chern class.
Considering the inclusion $\cE_2/\cE_1 \subset \cE_k/\cE_1$ and repeating the same argument as above,
we deduce that $(\cE_2/\cE_1)^{**}$ is a Hermitian flat locally free sheaf on $X'$.
Iterating this procedure, we see that $(\cE_i/\cE_{i-1})^{**}$ is a Hermitian flat locally free sheaf on $X'$
for every $i$.
\end{step}

\begin{step}
In this step, we prove that $\cE_k^{**}$ is a numerically flat locally free sheaf on $X'$ by induction on $k$.
Note that when $k=1$, there is nothing to prove.
Indeed, in this case, the sheaf $\cE_1=\nu^{[*]}\mathcal{E}$ is stable, and hence $\cE_1^{**}$ is a Hermitian flat locally free sheaf by Theorem~\ref{thm-stable}.

First, we claim that $\cE_i$ is pseudo-effective and satisfies $c_1(\cE_i)=0$ for every $i$.
The inclusion $\cE_i\subset \cE_k=\nu^{[*]}\cE$ induces the generically surjective morphism
$\cE_k^{*}\to \cE_i^{*}$.
By Propositions \ref{genersurjet} and \ref{dualpsef}, we see that $\cE_i^{*}$ is pseudo-effective.
Hence, once we know that $c_1(\cE_i)=0$, the sheaf $\cE_i$ is pseudo-effective by Proposition~\ref{dualpsef}.
To check $c_1(\cE_i)=0$, we consider the exact sequence
\begin{equation}\label{eq-jhshort}
0 \to \cE_{i-1} \to \cE_i \to \cE_i/\cE_{i-1} \to 0.
\end{equation}
Since $c_{1}(\cE_i/\cE_{i-1}) =0$, using \eqref{eq-jhshort}, 
we obtain inductively
\[
c_1(\cE_i)=c_1(\cE_{i-1})+c_1(\cE_i/\cE_{i-1})=0,
\]
starting from $c_1(\cE_1)=0$.

By considering \eqref{eq-jh} up to $(k-1)$,
the induction hypothesis shows that $\cE_{k-1}^{**}$ is a numerically flat locally free sheaf on $X'$.
Let $\mathcal{S}$ be the saturation of $\cE_{k-1}^{**}\subset \cE_k^{**}$, and set $\cF:=\cE_k^{**}$.
Then, we have the exact sequence
\[
0 \to \mathcal{S} \to \cF \to \mathcal{Q} \to 0.
\]
This operation does not change \eqref{eq-jhshort} in codimension $1$.
Indeed, by \cite[Lemma~2.9]{IMZ26} (or by the proof of Proposition~\ref{prop-ref}),
the quotient $\cE_{k}/\cE_{k-1}$ is locally free on $X'_{0}:=X'_{\reg}\cap X'_{\nu^{[*]}\mathcal{E}}$,
and \eqref{eq-jhshort} is exact as a sequence of vector bundles on $X'_{0}$.
Thus, the sheaf $\mathcal{S}$ coincides with $\cE_{k-1}^{**}$; in particular, it is locally free.
Moreover, we have $c_1(\mathcal{Q})=c_1(\cE_k/\cE_{k-1})=0$.

By Proposition~\ref{prop-ref}, there exists a Zariski closed subset $Z\subset X'$ of codimension $\ge 3$
such that $\mathcal{Q}$ is reflexive on $X'\setminus Z$.
Since $\mathcal{Q}^{**}=(\cE_k/\cE_{k-1})^{**}$ is locally free, the sheaf $\mathcal{Q}$ is locally free on $X'\setminus Z$.
Consequently, the sheaf $\cF$ is locally free on $X'\setminus Z$, and the sequence is exact as a sequence of vector bundles on $X'\setminus Z$.
Therefore, the locally free sheaf $\cF|_{X'\setminus Z}$ is determined by an extension class
\[
\beta \in H^{1}(X'\setminus Z,\Hom(\mathcal{Q}^{},\mathcal{S}))
\cong H^{1}(X'\setminus Z,\Hom(\mathcal{Q}^{**},\mathcal{S}^{**})).
\]

By  $\codim Z\ge 3$ and \cite[Theorem~1.14]{ST71},  the restriction map
\[
H^{1}\bigl(X',\Hom(\mathcal{Q}^{**},\mathcal{S}^{**})\bigr)\longrightarrow
H^{1}\bigl(X'\setminus Z,\Hom(\mathcal{Q}^{**},\mathcal{S}^{**})\bigr)
\]
is an isomorphism.
Let $\pi \colon \widetilde{X}' \to X'$ be a desingularisation.
Since $X'$ has rational singularities, for any locally free sheaf $V$ on $X'$ we have
\[
R^{j}\pi_{*}(\pi^{*}V)=V\otimes R^{j}\pi_{*}\mathcal{O}_{\widetilde{X}'}=0
\quad (j\ge 1).
\]
By the Leray spectral sequence, we obtain
\[
H^{1}\bigl(\widetilde{X}',\pi^{*}\Hom(\mathcal{Q}^{**},\mathcal{S}^{**})\bigr)
\cong
H^{1}\bigl(X',\Hom(\mathcal{Q}^{**},\mathcal{S}^{**})\bigr)
\cong
H^{1}\bigl(X'\setminus Z,\Hom(\mathcal{Q}^{**},\mathcal{S}^{**})\bigr).
\]
Hence, the class $\beta$ lifts to a class
\[
\widetilde{\beta}\in H^{1}\bigl(\widetilde{X}',\pi^{*}\Hom(\mathcal{Q}^{**},\mathcal{S}^{**})\bigr)
\cong H^{1}\bigl(\widetilde{X}',\Hom(\pi^{*}\mathcal{Q}^{**}, \pi^{*}\mathcal{S}^{**})\bigr).
\]
Let $V_{\widetilde{\beta}}$ be the locally free sheaf on $\widetilde{X}'$ defined by $\widetilde{\beta}$.
Since $\widetilde{X}'$ is a compact K\"ahler manifold and both $\pi^{*}\mathcal{Q}^{**}$ and $\pi^{*}\mathcal{S}^{**}$ are flat,
Simpson's result \cite{Sim92} (see also \cite[Theorem~1.1]{Den21}) implies that $V_{\widetilde{\beta}}$ is flat.
By construction, the sheaf $\pi^{*}\cF$ coincides with $V_{\widetilde{\beta}}$ over $X'\setminus Z$.
Therefore, the sheaf $\cF$ is flat on $X'_{\reg}$, and hence flat on $X'$ by Theorem~\ref{thm-maxqetale}~(2).
By the local freeness, \cite[Corollary]{Wu22} shows that $\cE_k^{**}=\cF$ is a numerically flat locally free sheaf on $X'$.
\end{step}
\end{proof}

We finally prove assertion (ii) of Theorem \ref{thm-main}. 

\begin{proof}[Proof of Theorem~\ref{thm-main}~(ii)]

\setcounter{step}{0}
\begin{step}
The basic strategy is to descend the flat connection on $\nu^{[*]}\mathcal{E}$ to a flat connection on $\cE|_{X_{\reg}}$.
However, this is not straightforward.
To perform this descent to $X_{\reg}$, we use the explicit construction of flat connections given in \cite{Den21}.

We consider
$$
f \colon \tilde{X}' \xrightarrow{\quad \pi \quad }  X' \xrightarrow{\quad  \nu \quad } X,
$$
where $\nu \colon X' \to X$ is a maximally quasi-\'etale cover and $\pi \colon \tilde{X}' \to X'$ is a resolution of singularities
which is an isomorphism over $X'_{\reg}$.
In what follows, we often identify $f^{-1}(X_{\reg}) \subset \tilde{X}'$ with $\nu^{-1}(X_{\reg}) \subset X'$
via $\pi \colon \tilde{X}' \to X'$.
Further, we consider the Jordan--H\"older filtration of $\mathcal{E}$:
$$
0 =: \cE_0 \subset \cE_1 \subset \cE_2 \subset \cdots \subset \cE_k :=\mathcal{E}.
$$
Arguing as in the proof of~(i), we see that
$$
F_i := (\cE_{i+1}/ \cE_i)^{**}
$$
is Hermitian flat on $X_{\reg}$.
Let $h_{i}$ be a Hermitian flat metric on $F_{i}|_{X_{\reg}}$.
Set $F'_{i}:=\nu^{[*]} F_{i}$ (which is locally free on $X'$ by Theorem~\ref{thm-maxqetale}~(2)),
and let $h'_{i}$ be the extension of $\nu^{*}h_{i}$ to a Hermitian flat metric on $F'_{i}$.
Moreover, by~(i), the reflexive pullback $f^{[*]}\cE$ is locally free and can be written as a successive extension
of the Hermitian flat vector bundles $(\tilde{F_i}, \tilde{h_i})$,
where
$$(\tilde{F_i}, \tilde{h_i}) := \pi^*(F'_i, h'_i).$$
By \cite{Sim92}, we already know that $f^{[*]}\cE$ admits a flat connection.
However, in order to descend it to a flat connection on $\mathcal{E}|_{X_{\reg}}$,
we need the explicit construction of the flat connection in \cite{Den21}.

First, consider the vector bundle $\oplus_{i=1}^k \tilde{F_i}$
equipped with the connection $\tilde \nabla$
and the complex structure $\bar{\partial}_{\eta}$ defined by
\begin{equation*}
\tilde \nabla := \oplus_{i=1}^k D^{(1,0)} _{\tilde{h_i}}
=
\begin{pmatrix}
D^{(1,0)} _{\tilde{h_1}} & 0 & \cdots & 0  \\
0 & D^{(1,0)} _{\tilde{h_2}} & \cdots & 0\\
\vdots & \vdots & \ddots & \vdots \\
0 & 0 & \cdots & D^{(1,0)} _{\tilde{h_k}}
\end{pmatrix}
\end{equation*}
and
\begin{equation}\label{eq:dbar-eta}
\bar{\partial}_{\eta} =
\begin{pmatrix}
\dbar & \eta_{1,2} & \cdots & \eta_{1,k}  \\
0 &  \dbar & \ddots & \vdots \\
\vdots &  \ddots & \dbar & \eta_{k-1,k}  \\
0 &  \cdots & 0 & \dbar
\end{pmatrix}.
\end{equation}
Here $D^{(1,0)} _{\tilde{h_i}}$ denotes the $(1,0)$-part of the Chern connection defined by $\tilde{h_i}$.
Our purpose is to construct smooth $\Hom(\tilde F_j, \tilde F_i)$-valued $(0,1)$-forms
\[
\eta_{i,j} \in C^\infty _{(0,1)} (\tilde{X}', \Hom(\tilde F_j, \tilde F_i))
\]
such that:
\begin{itemize}
\item[(1)] The holomorphic vector bundle $(\oplus_{i=1}^k \tilde{F_i}, \dbar_\eta)$ is isomorphic to $f^{[*]}\cE$ on $\tilde{X}'$.

\item[(2)]
The following equalities hold:
$$
D^{(1,0)}_{\tilde{h}_j^{*}\otimes \tilde{h}_i}\,\eta_{i,j}=0
\quad \text{ and } \quad  \dbar \eta_{i,j}=0,
$$
where $D^{(1,0)}_{\tilde{h}_j^{*}\otimes \tilde{h}_i}$ is
the $(1,0)$-part of the Chern connection on $\Hom(\tilde F_j,\tilde F_i)$ defined by $\tilde{h}_j$ and $\tilde{h}_i$.

\item[(3)] $\eta_{i,j} |_{f^{-1}(X_{\reg})}$ can be written as the pullback of a form on $X_{\reg}$.
\end{itemize}
Once (1)--(3) are established, we obtain a flat connection on $\cE|_{X_{\reg}}$.
Indeed, writing $\eta_{i,j}=f^{*}\xi_{i,j}$ over $X_{\reg}$ by (3),
the holomorphic vector bundle $\oplus_{i=1}^k  F_i$ with $\bar{\partial}_{\xi}$ defined as in \eqref{eq:dbar-eta}
is isomorphic to $\cE|_{X_{\reg}}$ by (1).
Moreover, the connection
$$
\oplus_{i=1}^k D^{(1,0)} _{ h_i}+\dbar_{\xi}
$$
is flat by (2).
\end{step}

\begin{step}
As a warm-up, we treat the case $k=2$.
Consider the exact sequence
\begin{align}\label{eq-short-pull}
0 \to \tilde F_{1} \to f^{[*]}\mathcal{E}_{2}=f^{[*]}\mathcal{E} \to \tilde F_{2}  \to 0. 
\end{align}
We will construct $\eta_{1,2} \in C^\infty _{(0,1)} (\tilde{X}', \Hom (\tilde F_{2}, \tilde F_{1}))$
satisfying properties $(1)$, $(2)$ and $(3)$.

The extension class of \eqref{eq-short-pull} can be represented by a form
$$
\alpha \in  C^\infty _{(0,1)} (\tilde{X}', \Hom (\tilde F_{2}, \tilde{F}_1))
$$
such that
\[
\bar{\partial}_{\alpha} =
\begin{pmatrix}
\dbar & \alpha \\
0 &  \dbar
\end{pmatrix}
\]
satisfies the property $(1)_{\alpha}$, obtained from $(1)$ by replacing $\eta_{i,j}$ with $\alpha$.
Since $\dbar \alpha =0$ by $\bar{\partial}_{\alpha} ^2 =0$,
Hodge theory provides $\psi\in C^\infty (\tilde{X}', \Hom (\tilde F_{2}, \tilde{F}_1))$ such that
\begin{equation}\label{eq-eta12-hodge}
\eta_{1,2} := \alpha +\dbar \psi \in \ker D^{(1,0)}_{\tilde h_2^{*} \otimes \tilde h_1}.
\end{equation}
Hence, the form $\eta_{1,2}$ satisfies the desired property $(1)$ and 
\begin{align}\label{eq-dprime}
D^{(1,0)}_{\tilde h_2^{*} \otimes \tilde h_1}\eta_{1,2}=0.
\end{align}
It remains to verify the descent property (3).

Let $\gamma \in C^\infty _{(0,1)}(X_{\reg}, \Hom (F_2, F_1))$ be a smooth representative of the extension class of $\cE_2|_{X_{\reg}}$ defined by
\begin{align}\label{eq-short}
0 \to F_{1} \to {\cE_2}|_{ X_{\reg} }  \to F_{2}  \to 0.
\end{align}
Restricting \eqref{eq-short-pull} to $f^{-1}(X_{\reg})\cong \nu^{-1}(X_{\reg})$
and comparing with the pullback of \eqref{eq-short},
we see that $\nu^*\gamma$ and $\eta_{1,2}|_{\nu^{-1}(X_{\reg})}$ represent the same extension class over $X_{\reg}$.
Hence, there exists $\phi\in C^\infty (\nu^{-1}(X_{\reg}),  \Hom (F'_{2}, F'_{1}))$ such that
\begin{equation}\label{eq-eta12-splitting}
\eta_{1,2}|_{\nu^{-1}(X_{\reg})} = \nu^*\gamma +\dbar \phi \quad\text{over } X_{\reg}.
\end{equation}

We may assume that $\nu \colon X' \to X$ is a Galois cover with Galois group $G$.
The section defined by the average
$$
\frac{1}{\deg \nu} \sum_{g \in G} g^{*} \phi \in C^\infty (\nu^{-1}(X_{\reg}), \Hom (F'_2, F'_1))
$$
is $G$-invariant, hence it descends to a smooth section
$\phi_{\rm av}\in C^\infty (X_{\reg}, \Hom (F_2, F_1))$ satisfying
$$
\nu^*\phi_{\rm av}=\frac{1}{\deg \nu} \sum_{g \in G} g^{*} \phi \quad\text{over } X_{\reg}.
$$
Set
\begin{equation*}\label{eq-s12}
s:=\nu^*\phi_{\rm av}-\phi \in C^\infty (\nu^{-1}(X_{\reg}), \Hom (F'_2, F'_1)).
\end{equation*}
Then, by Lemma \ref{plurharmonic} (which will be proved in the next step),
the section $s$ is holomorphic over $X_{\reg}$ and bounded with respect to $h_2^{\prime *}\otimes h_1'$.
The Riemann extension type theorem shows that $\pi^*s$ extends to a holomorphic section of $\Hom(\tilde F_2,\tilde F_1)$ on $\tilde{X}'$.

Since $\dbar s=0$ on $\nu^{-1}(X_{\reg})$, we obtain
\begin{equation}\label{eq-eta12-descent}
\eta_{1,2}
= \nu^*\gamma + \dbar \phi
= \nu^*\gamma + \dbar(\nu^*\phi_{\rm av}-s)
= \nu^*(\gamma+\dbar \phi_{\rm av})
\quad\text{on } \nu^{-1}(X_{\reg})
\end{equation}
from \eqref{eq-eta12-splitting}.
Thus $\eta_{1,2}$ is the pullback of a form over $X_{\reg}$, proving (3) in the case $k=2$.
\end{step}

\begin{step}
In this step, we prove the following lemma.
\begin{lemma}\label{plurharmonic}
Let us consider the same situation as in Step~2.
Then the section $s=\nu^*\phi_{\rm{av}}-\phi$ is holomorphic over $X_{\reg}$ and bounded with respect to $h_2^{\prime *}\otimes h_1'$.
\end{lemma}

\begin{proof}
We first show that
\begin{equation}\label{flatcond}
D^{(1,0)}_{h_{2}^{*}\otimes h_{1}} (\gamma +\dbar \phi_{\rm{av}}) =0,
\end{equation}
where $D^{(1,0)}_{h_{2}^{*}\otimes h_{1}}$ is the $(1,0)$-part of the Chern connection on $\Hom(F_2,F_1)$ defined by $h_2$ and $h_1$.
Let $U$ be a small open subset of $X_{\reg}$ such that
$\nu^{-1} (U) =\sqcup_{\ell=1}^{\deg \nu} U_\ell$ is a disjoint union of open subsets,
and $\nu|_{U_{\ell}}: U_\ell \to U$ is biholomorphic.
Fix $U_1$ and, for each $1\leq \ell \leq \deg\nu$, choose $g_{\ell} \in G$ such that $g_{\ell}\colon U_1 \to U_\ell$ is biholomorphic.
Since $h'_i=\nu^{*}h_{i}$, we have $g_{\ell}^* h'_{i} = h'_{i}$ on $U_1$.
By construction, we have $D^{(1,0)}_{h_2^{\prime *}\otimes h_1'}(\eta_{1,2})=0$ on $\nu^{-1}(X_{\reg})$ by \eqref{eq-eta12-hodge},
and hence, we obtain
\begin{equation}\label{eq-eta}
0= D^{(1,0)}_{h_2^{\prime *}\otimes h_1'} (\nu^* \gamma +\dbar \phi ) \quad\text{on }U_1 
\end{equation}
by \eqref{eq-eta12-splitting}. 
Applying $g_{\ell}^*$ to \eqref{eq-eta} and using $g_{\ell}^*(\nu^*\gamma)=\nu^*\gamma$, we obtain
$$
0= D^{(1,0)}_{h_2^{\prime *}\otimes h_1'} \bigl(\nu^* \gamma +\dbar (g_{\ell}^*\phi)\bigr)
\quad\text{on }U_1.
$$
Averaging over $\ell$ yields
$$
0= D^{(1,0)}_{h_2^{\prime *}\otimes h_1'} \Bigl(\nu^* \gamma +\dbar  \frac{1}{\deg \nu} \sum_{\ell=1}^{\deg\nu} g_{\ell}^{* }  \phi\Bigr)
=
D^{(1,0)}_{h_2^{\prime *}\otimes h_1'} \bigl(\nu^* (\gamma+ \dbar \phi_{\rm{av}})\bigr)
\quad\text{on } U_1.
$$
Since $\nu$ is \'etale and $h_i'$ is defined as the pullback of $h_{i}$,
we obtain \eqref{flatcond} on $U$.

Now put $s:=\nu^*\phi_{\rm{av}}-\phi$.
From \eqref{eq-dprime}, \eqref{eq-eta12-splitting}, and \eqref{flatcond}, we have
\begin{align*}
D^{(1,0)}_{h_2^{\prime *}\otimes h_1'}\dbar s &\stackrel{\eqref{eq-eta12-splitting}}{=}
D^{(1,0)}_{h_2^{\prime *}\otimes h_1'} \dbar \nu^*\phi_{\rm{av}} +D^{(1,0)}_{h_2^{\prime *}\otimes h_1'}  (\nu^{*}\gamma - \eta_{1,2}) \\
&\stackrel{\eqref{eq-dprime}}{=}D^{(1,0)}_{h_2^{\prime *}\otimes h_1'} (\dbar \nu^*\phi_{\rm{av}} +  \nu^{*}\gamma) \\
& \stackrel{\qquad }{=} \nu^{*}\big( D^{(1,0)}_{h_2^{*}\otimes h_1} \dbar \phi_{\rm{av}} + \gamma \big) \stackrel{\eqref{flatcond}}{=} 0 \quad\text{over } X_{\reg}.
\end{align*}
Together with the Hermitian flatness of $h'_2$ and $h'_1$, we obtain
\begin{align*}
&\idelbar  |\,s\,|^{2}_{\,h_2^{\prime *}\otimes h_1'}\\
=
&\sqrt{-1}\,\Bigl\langle D^{(1,0)}_{h_2^{\prime *}\otimes h_1'}  s, \, D^{(1,0)}_{h_2^{\prime *}\otimes h_1'} s \Bigr\rangle_{\,h_2^{\prime *}\otimes h_1'} +
\sqrt{-1}\,\Bigl\langle \dbar  s, \, \dbar  s \Bigr\rangle_{\,h_2^{\prime *}\otimes h_1'}
\ \ge\ 0
\quad\text{over } X_{\reg}.
\end{align*}
In particular, the function $|s|^{2}_{h_2^{\prime *}\otimes h_1'}$ is plurisubharmonic over $X_{\reg}$, and hence on $X'$.
Since $X'$ is compact, it is constant, and therefore $\dbar s=0$ and $ D^{(1,0)}_{h_2^{\prime *}\otimes h_1'}  s=0$ over $X_{\reg}$.
This proves the lemma.
\end{proof}
\end{step}

\begin{step}
In this step, we construct $\eta_{i,j} \in C^\infty _{(0,1)} (\tilde{X}', \Hom (\tilde F_j, \tilde F_i))$ satisfying $(1)$--$(3)$ by induction on $k$.
By the induction hypothesis, for every $i<j \leq k-1$, there exists $\eta_{i,j}$ satisfying $(2)$ and $(3)$ such that
$(\oplus_{i=1}^{k-1} \tilde{F_i}, \bar{\partial}_{k-1})$ is isomorphic to $f^{[*]} \cE_{k-1}$ on $\tilde{X}'$,
where $\bar{\partial}_{k-1}$ is the complex structure defined by
\[
\bar{\partial}_{k-1} =
\begin{pmatrix}
\dbar & \eta_{1,2} & \cdots & \eta_{1,k-1}  \\
0 &  \dbar & \cdots & \eta_{2,k-1}\\
0 &  0 & \cdots & \cdots  \\
0 &  0 & 0 & \dbar
\end{pmatrix}.
\]
Our purpose is to construct $\{\eta_{i,k}\}_{i=1}^{k-1}$ satisfying $(1)$, $(2)$, and $(3)$.

Representing the extension class of the corresponding exact sequence, we choose smooth forms
$\{\alpha_i\}_{i=1}^{k-1} \subset C^\infty _{(0,1)} (\tilde{X}', \Hom (\tilde F_k, \tilde F_i))$
such that $(\oplus_{i=1}^k \tilde{F_i}, \bar{\partial}_\alpha)$ is isomorphic to $f^{[*]} \cE$ on $\tilde{X}'$, where
\[
\bar{\partial}_{\alpha} =
\begin{pmatrix}
\dbar & \eta_{1,2} & \cdots & \eta_{1,k-1} & \alpha_1 \\
0 &  \dbar & \cdots & \eta_{2,k-1} & \alpha_2\\
0 &  0 & \cdots & \cdots  &\cdots \\
0 &  0 & 0 & \dbar & \alpha_{k-1}\\
0 &  0 & 0 & 0 & \dbar
\end{pmatrix}.
\]
Then, we have the following two properties:

(i) Since $\bar{\partial}_{\alpha}^{2}=0$, we obtain
\begin{align}\label{complexstr}
\dbar \alpha_{k-1} = 0 \text{ and }
\dbar \alpha_{i} = - \sum_{j=1}^{k-1-i} \eta_{i,i+j}\wedge \alpha_{i+j}
\quad \text{for every } i<k-1.
\end{align}

(ii) For smooth forms $\{\beta_i\}_{i=1}^{k-1} \subset C^\infty _{(0,1)} (\tilde{X}', \Hom (\tilde F_k, \tilde F_i))$ on $\tilde{X}'$,
we set the new complex structure
\[
\bar{\partial}_{\beta} =
\begin{pmatrix}
\dbar & \eta_{1,2} & \cdots & \eta_{1,k-1} & \beta_1 \\
0 & \dbar & \cdots & \eta_{2,k-1} & \beta_2\\
0 & 0 & \cdots & \cdots & \cdots \\
0 & 0 & 0 & \dbar & \beta_{k-1}\\
0 & 0 & 0 & 0 & \dbar
\end{pmatrix}.
\]
The operator $\bar{\partial}_\beta$ defines the same vector bundle as $\bar{\partial}_\alpha$
if and only if there are smooth sections
\(f_i\in C^\infty\bigl(\tilde{X}',\,\Hom(\tilde F_k,\tilde{F_i})\bigr)\) such that
\begin{equation}\label{eq:beta-alpha}
\left\{
\begin{aligned}
\beta_{k-1} &= \alpha_{k-1} + \dbar f_{k-1}, \\
\beta_{k-2} &= \alpha_{k-2} - \eta_{k-2,k-1}\cdot f_{k-1} + \dbar f_{k-2}, \\
&\ \ \vdots \\
\beta_{1}   &= \alpha_{1} - \sum_{2\le j\le k-1}\eta_{1,j}\cdot f_j + \dbar f_1.
\end{aligned}
\right.
\end{equation}

\medskip

We now construct $\eta_{k-1,k} \in C^\infty _{(0,1)} (\tilde{X}',  \Hom (\tilde F_k, \tilde F_{k-1}))$
satisfying $(2)$ and $(3)$.
By \eqref{complexstr}, we have $\dbar \alpha_{k-1}=0$.
Thus, by Hodge theory on $\tilde{X}'$, we can find $\psi_{k-1}\in C^\infty(\tilde{X}',\Hom(\tilde F_k,\tilde F_{k-1}))$ such that
$$
\alpha_{k-1} +\dbar \psi_{k-1}\in \ker D^{(1,0)}_{\tilde h_k^{*}\otimes \tilde h_{k-1}}. 
$$
By the same argument as in Step~2 (using Lemma~\ref{plurharmonic} on $\nu^{-1}(X_{\reg}) \cong  f^{-1}(X_{\reg})$), there exists $s_{k-1} \in H^0 (\tilde{X}',  \Hom (\tilde F_k, \tilde F_{k-1}))$ such that
$$
\eta_{k-1, k} := \alpha_{k-1} +\dbar (\psi_{k-1}+s_{k-1}) \in \ker D^{(1,0)}_{\tilde h_k^{*}\otimes \tilde h_{k-1}},
$$
and $\eta_{k-1, k}$ can be written as the pullback of a form on $X_{\reg}$.
Then $\eta_{k-1, k}$ satisfies $(2)$ and $(3)$.

Set $f_{k-1}:=\psi_{k-1}+s_{k-1}$ and define $(\beta_1, \cdots,  \beta_{k-2})$ by
$$
(\beta_1, \cdots,  \beta_{k-2})
:=
\bigl( \alpha_{1} - \eta_{1,k-1}\cdot f_{k-1}, \cdots,  \alpha_{k-2}- \eta_{k-2,k-1}\cdot f_{k-1}  \bigr).
$$
Consider the complex structure
\[
\bar{\partial}_{\beta, \eta} =
\begin{pmatrix}
\dbar & \eta_{1,2} & \cdots & \eta_{1,k-1} & \beta_1 \\
0 &  \dbar & \cdots & \eta_{2,k-1} & \beta_2\\
0 &  0 & \cdots & \cdots  &\cdots \\
0 &  0 & \cdots & \cdots  & \beta_{k-2} \\
0 &  0 & 0 & \dbar & \eta_{k-1,k}\\
0 &  0 & 0 & 0 & \dbar
\end{pmatrix}.
\]
By \eqref{eq:beta-alpha}, $(\oplus_{i=1}^k \tilde{F_i}, \bar{\partial}_{\beta,\eta})$ is isomorphic to $(\oplus_{i=1}^k \tilde{F_i}, \bar{\partial}_\alpha)$,
hence to $f^{[*]}\cE$ on $\tilde{X}'$.

\medskip

We now construct $\eta_{k-2,k} \in C^\infty _{(0,1)} (\tilde{X}',  \Hom (\tilde F_k, \tilde F_{k-2}))$
satisfying $(2)$ and $(3)$.
Applying \eqref{complexstr} to $\bar{\partial}_{\beta, \eta}$, we have
$$
\dbar (\beta_{k-2}) = - \eta_{k-2, k-1} \wedge \eta_{k-1, k} \quad\text{ on } \tilde{X}'.
$$
The right-hand side is $D^{(1,0)}_{\tilde h_k^{*}\otimes \tilde h_{k-2}}$-closed.
On the other hand, by Hodge theory, we can find $\psi_{k-2}\in C^\infty(\tilde{X}',\Hom(\tilde F_k,\tilde F_{k-2}))$ such that
$$
\beta_{k-2} + \dbar \psi_{k-2}\in \ker D^{(1,0)}_{\tilde h_k^{*}\otimes \tilde h_{k-2}}.
$$
Moreover, since $\eta_{k-2, k-1} \wedge \eta_{k-1, k}$ is the pullback of a form over $X_{\reg}$,
the same argument as in Step~2 yields a section $s_{k-2} \in H^0 (\tilde{X}',  \Hom (\tilde F_k, \tilde F_{k-2}))$ such that
$$
\eta_{k-2,k} := \beta_{k-2} + \dbar (\psi_{k-2}+s_{k-2})
$$
satisfies $(2)$ and $(3)$.

Set $f_{k-2}:=\psi_{k-2}+s_{k-2}$ and define $(\gamma_1, \cdots,  \gamma_{k-3})$ by
$$
(\gamma_1, \cdots,  \gamma_{k-3})
:=\bigl(\beta_{1} - \eta_{1,k-2}\cdot f_{k-2}, \cdots,  \beta_{k-3}- \eta_{k-3,k-2}\cdot f_{k-2}  \bigr).
$$
Consider the complex structure
\[
\bar{\partial}_{\gamma, \eta} =
\begin{pmatrix}
\dbar & \eta_{1,2} & \cdots & \eta_{1,k-1} & \gamma_1 \\
0 &  \dbar & \cdots & \eta_{2,k-1} & \gamma_2\\
0 &  0 & \cdots & \cdots  &\cdots \\
0 &  0 & \cdots & \cdots  &\gamma_{k-3} \\
0 &  0 & \cdots & \cdots  & \eta_{k-2,k} \\
0 &  0 & 0 & \dbar & \eta_{k-1,k}\\
0 &  0 & 0 & 0 & \dbar
\end{pmatrix}.
\]
By \eqref{eq:beta-alpha}, $(\oplus_{i=1}^k \tilde{F_i}, \bar{\partial}_{\gamma, \eta})$ is isomorphic to $(\oplus_{i=1}^k \tilde{F_i}, \bar{\partial}_{\beta,\eta})$,
hence to $f^{[*]}\cE$ on $\tilde{X}'$.
\medskip

Repeating this procedure, we construct $\eta_{k-3,k},\ldots,\eta_{1,k}$ satisfying $(2)$ and $(3)$ such that $\bar{\partial}_\eta$ satisfies $(1)$.
\end{step}
\end{proof}

\section{Numerically flat vector bundles on compact complex manifolds}\label{nonkahler}

The aim of this section is to extend Theorem \ref{thm-stable} in a different direction, namely to the case where $X$ is not necessarily K\"ahler. This gives a positive answer to the question raised in \cite[Remark 1.21]{DPS94}.

We first recall some basic notions concerning stability of sheaves on compact complex manifolds. 
We refer, for example, to \cite[Section 6]{CP25} for more details. 
Let $X$ be a compact complex manifold of dimension $n$. 
A Hermitian form $\omega_{g}$ on $X$ is  said to be a \textit{Gauduchon metric} if
$$
\ddbar \omega_{g}^{n-1}=0.
$$
The Bott--Chern cohomology group $H^{p,q}_{BC}(X)$ and 
the Aeppli cohomology group $H^{p,q}_{A}(X)$ are defined as follows: 
\begin{align*}
H^{p,q}_{BC}(X):&=
\frac{
(\ker \partial \colon \mathcal C^{p,q}(X)\to \mathcal C^{p+1,q}(X))
\cap
(\ker \dbar  \colon \mathcal C^{p,q}(X)\to \mathcal C^{p,q+1}(X))
}{
\Imm \ddbar \colon \mathcal C^{p-1,q-1}(X)\to \mathcal C^{p,q}(X)
}, \\
H^{p,q}_{A}(X):&=
\frac{
\ker \ddbar \colon \mathcal C^{p,q}(X)\to \mathcal C^{p+1,q+1}(X)
}{
(\Imm \partial \colon \mathcal C^{p-1,q}(X)\to \mathcal C^{p,q}(X))
+
(\Imm \dbar \colon \mathcal C^{p,q-1}(X)\to \mathcal C^{p,q}(X))
}.
\end{align*}
The movable cone 
$$\Mov(X)\subset H^{n-1,n-1}_{A}(X,\mathbb R)$$
is the closed cone generated by the classes $[\omega_{g}^{n-1}]$, 
where $\omega_{g}$ ranges over all Gauduchon metrics on $X$. 
The pseudo-effective cone
$$
\Psef(X)\subset H^{1,1}_{BC}(X,\mathbb R)
$$
is the closed cone consisting of classes represented by $d$-closed positive $(1,1)$-currents on $X$.
We next recall some basic properties of stability for torsion-free coherent sheaves on $X$.

\begin{defn}
Let $X$ be a compact complex manifold of dimension $n$, and let $\cE$ be a coherent torsion-free sheaf on $X$. 
Let $\omega_{g}$ be a Gauduchon metric and set $\alpha:=\omega_{g}^{n-1}$. 
We define the slope of $\cE$ with respect to $\alpha$ by
$$
\mu_\alpha(\cE):=
\frac{1}{\rank \cE}
\int_X c_1(\cE)\wedge \omega_{g}^{n-1}.
$$
We say that $\cE$ is stable with respect to $\alpha$ if
\[
\mu_\alpha(\cF) < \mu_\alpha(\cE)
\]
for every subsheaf $\cF \subset \cE$ satisfying
$
1 \leq \rk \cF < \rk \cE$.
\end{defn}

Then, we prove the following lemma.

\begin{lemma}\label{nonkahlerherm}
Let $X$ be a compact complex manifold equipped with a Gauduchon metric $\omega_{g}$.
Let $E$ be a numerically flat vector bundle on $X$, namely, assume that
$c_1(E)=0$ in $H^{1,1}_{BC}(X,\mathbb R)$ and that $E$ is nef.
If $E$ is stable with respect to $\alpha:=\omega_{g}^{n-1}$, then $E$ is Hermitian flat.
\end{lemma}

\begin{proof}
The proof follows the same idea as that of Theorem \ref{thm-stable}. 
Recall that $E$ is nef if and only if for every $\varepsilon>0$, there exists a smooth metric $h_\varepsilon$ on 
$L:=\mathcal{O}_{\mathbb P(E)}(1)$ such that
\[
\sqrt{-1}\Theta_{h_\varepsilon}\bigl(\mathcal{O}_{\mathbb P(E)}(1)\bigr)\ge -\varepsilon\,\omega_{g'}.
\]
Here $\omega_{g'}$ is a fixed Hermitian metric on $\mathbb P(E)$. 
Let $r$ be the rank of $E$. Consider the natural projection
$
\pi \colon \mathbb P(E)\to X.
$ and the hyperplane bundle 
$
L=\mathcal O_{\mathbb P(E)}(1).
$
Since $E$ is nef and $L$ is relatively ample, there exists a sequence of smooth Hermitian metrics $h_\varepsilon$ on $L$ such that
$$
\sqrt{-1}\Theta_{h_\varepsilon}(L)\ge -\varepsilon \pi^* \omega_{g}.
$$
Fix a Gauduchon metric $\omega_{g'}$ on $\mathbb P(E)$. Since $\mathbb P(E)$ is compact, the quantities
$$
\int_{\mathbb P(E)}
\bigl(\sqrt{-1}\Theta_{h_\varepsilon}(L)+\varepsilon\pi^*\omega_{g}\bigr)
\wedge \omega_{g'}^{n+r-2}
$$
are uniformly bounded.
After a suitable normalization and after passing to a subsequence, we may assume that
$\sqrt{-1}\Theta_{h_\varepsilon}(L)$ weakly converges 
to a positive current $T \geq 0$  such that 
$$
[T]=c_1(L)\in H^{1,1}_{BC}(\mathbb P(E),\mathbb R).
$$
Let $h$ be a singular metric on $L$ such that
$
T=\sqrt{-1}\Theta_h(L).
$
Set
\[
P_{+}:=\{x\in \mathbb P(E)\mid \nu(T,x)>0\}.
\]

We divide the argument into two cases.

\setcounter{case}{0}
\begin{case}
Suppose that $P_{+}$ is not dominant over $X$.
By applying the positivity theorem for direct image sheaves \cite{PT18}, the metric $h^{r+1}$ induces an $L^2$ metric $H$ on
\[
\det E\otimes E
=
p_*\bigl(
\mathcal O_P(K_{\mathbb P(E)/X}+(r+1)L)
\bigr)
\]
such that $(\det E\otimes E,H)$ is positively curved on $X$.
Since $c_1(E)=0$, it follows that $(\det E\otimes E,H)$ is Hermitian flat. Hence $E$ is also Hermitian flat.
\end{case}

\begin{case}
Suppose that $P_{+}$ is dominant over $X$.
We  show that this case cannot occur.
Since $E$ is stable and $X$ is smooth, 
by \cite{LY87}, there exists a smooth Hermitian--Einstein metric $G$ on $E$, namely,
\[
\sqrt{-1}\Theta_G(E)\wedge \omega_{g}^{\,n-1}\equiv 0
\qquad\text{on }X.
\]
The metric $G$ induces a quotient metric $h_1$ on $L$.

We now construct a destabilizing subsheaf.
By Proposition \ref{addlemma}, there exists $0<\delta<1$ such that
\[
\cG
=
\pi_*
\Bigl(
\mathcal O_P(K_{\mathbb P(E)/X}+(r+1)L)
\otimes
\mathcal I\bigl(
(h^\delta h_1^{1-\delta})^{r+1}
\bigr)
\Bigr)
\]
is a nonzero proper subsheaf of $E\otimes\det E$.
By construction, we have
\[
\sqrt{-1}\Theta_{h_1}(L)\wedge \pi^*\omega_{g}^{\,n-1}\ge 0
\qquad\text{on }\mathbb P(E).
\]
Together with $\sqrt{-1}\Theta_h(L)\ge 0$, this yields
\[
\sqrt{-1}\Theta_{h^\delta h_1^{1-\delta}}(L)
\wedge \pi^*\omega_{g}^{\,n-1}
\ge 0.
\]
Applying \cite[Theorem 1.6]{CCP19}, we obtain
\[
\sqrt{-1}\Theta_{\det H}(\det\cG)\wedge \omega_{g}^{\,n-1}\ge 0
\qquad\text{on }X,
\]
where $H$ denotes the induced $L^2$ metric on $\cG$.
Therefore, we obtain 
\[
\int_X c_1(\cG)\wedge \omega_{g}^{n-1}\ge 0.
\]
Since $\cG\otimes(\det E)^*$ is a nonzero proper subsheaf of $E$, it destabilizes $E$, contradicting the stability of $E$.
\end{case}
\end{proof}

Combining Lemma \ref{nonkahlerherm} with the argument of \cite[Section 2]{DPS94}, 
we obtain the main theorem of this section.

\begin{thm}\label{DPSnonkahler}
Let $X$ be a compact complex manifold, and let $E$ be a locally free sheaf on $X$.
If $E$ is numerically flat, then $E$ admits a filtration by holomorphic subbundles
\begin{equation}\label{filtra1}
\{0\}=E_0\subset E_1\subset\cdots\subset E_k=E
\end{equation}
such that each graded quotient $E_i/E_{i-1}$ is Hermitian flat.
In particular, all Chern classes $c_k(E)$ vanish.
\end{thm}

\begin{proof}
The proof is the same as that of Theorem~\ref{thm-main}, but is
considerably simpler in the present setting, where $X$ is smooth and $E$ is
locally free.
Fix a Gauduchon metric $\omega_{g}$. Since $E$ is numerically flat, we can show that $E$ is semistable with respect to $\omega_{g}^{n-1}$.
Let $\cF\subset \mathcal O_X(E)$ be a reflexive $\omega_{g}^{n-1}$-stable subsheaf. 
By the same argument as in \cite[Theorem 1.18, Step 1 and Lemma 1.20]{DPS94}, the sheaf $\cF$ is in fact a numerically flat subbundle of $E$. Note that the proof of \cite[Theorem 1.18, Step 1 and Lemma 1.20]{DPS94} does not require the K\"ahler assumption.
Applying Lemma \ref{nonkahlerherm} to $\cF$, we conclude that $\cF$ is Hermitian flat. The theorem then follows by induction on the rank, since $E/\cF$ is numerically flat again.
\end{proof}

It should be possible to generalize Theorem \ref{DPSnonkahler} to certain singular settings. However, the following question seems particularly interesting.

\begin{question}\label{flatconqq}
Let $X$ be a compact complex manifold, and let $E$ be a locally free sheaf on $X$.
If $E$ is numerically flat, does $E$ admit a holomorphic flat connection compatible with the filtration \eqref{filtra1}?
\end{question}

Recall that, when $X$ is K\"ahler, Question \ref{flatconqq} has an affirmative answer by \cite{Sim92,Den21}; 
the key tool is the $\ddbar$-lemma. It is unclear whether Question \ref{flatconqq} remains true for manifolds that do not satisfy the $\ddbar$-lemma.

\bibliographystyle{alpha}

\end{document}